%% file: arxiv.tex
\documentclass{article}
\pdfoutput=1
\usepackage[english]{babel}
\usepackage[utf8]{inputenc}
\usepackage[T1]{fontenc}
\usepackage[a4paper,top=3cm,bottom=2cm,left=3cm,right=3cm,marginparwidth=1.75cm]{geometry}
\input{macros_arxiv.tex}

\title{Perspectives on Stochastic Localization}
\author{Bobby Shi\thanks{University of Texas at Austin, \texttt{bhshi@utexas.edu} } 
\and 
Kevin Tian\thanks{University of Texas at Austin, \texttt{kjtian@cs.utexas.edu}}
\and
Matthew S.\ Zhang\thanks{University of Toronto, \texttt{matthew.zhang@mail.utoronto.ca}}
}
\date{}

\begin{document}
\maketitle

\begin{abstract}
We survey different perspectives on the \emph{stochastic localization} process of Eldan, a powerful construction that has had many exciting recent applications in high-dimensional probability and algorithm design. Unlike prior surveys on this topic, our focus is on giving a self-contained presentation of all known alternative constructions of Eldan's stochastic localization, with an emphasis on connections between different constructions. Our hope is that by collecting these perspectives, some of which had primarily arisen within a particular community (e.g., probability theory, theoretical computer science, information theory, or machine learning), we can broaden the accessibility of stochastic localization, and ease its future use.
\end{abstract}

\tableofcontents
\newpage

\input{sections/intro.tex}
\input{sections/measure.tex}
\input{sections/posterior.tex}
\input{sections/diffusion.tex}
\input{sections/renormalization.tex}

\input{sections/schrodinger.tex}
\input{sections/eot.tex}

\section*{Acknowledgements}
MSZ was supported by an NSERC CGS-D fellowship. 

We thank Sinho Chewi and Ronen Eldan for helpful discussions.

\newpage

\printbibliography

\appendix
\input{sections/anisotropic.tex}
\input{sections/rgo.tex}

\end{document}

%% file: macros_arxiv.tex
\usepackage{amsmath} 
\usepackage{amssymb}
\usepackage{arydshln}
\usepackage{amsthm}
\usepackage{algorithm}
\usepackage{algpseudocode}
\usepackage[
    style=alphabetic, 
    backref=true, 
    maxnames=99, 
    maxalphanames=6
]{biblatex}
\usepackage{bm} 
\usepackage{booktabs} 
\usepackage{csquotes}
\usepackage{enumitem} 
\usepackage{etoolbox} 
\usepackage{eucal} 
\usepackage[scaled=.92]{helvet} 
\usepackage[scaled=0.97]{inconsolata}
\usepackage{mathtools} 
\usepackage{mleftright}\mleftright 
\usepackage{microtype} 
\usepackage{parskip}
\usepackage{tikz-cd} 
\usepackage{thmtools}
\usepackage{thm-restate} 
\usepackage{url} 
\usepackage{verbatim}  
\usepackage{xcolor}
\usepackage{ninecolors}
\NineColors{saturation=high}
\definecolor{citegreen}{HTML}{208054}
\definecolor{citeblue}{HTML}{0055cc}
\usepackage[most]{tcolorbox}
\usepackage{hyperref}
\usepackage[capitalize]{cleveref} 
\newtheorem{theorem}{Theorem}
\newtheorem{proposition}{Proposition}
\newtheorem{lemma}{Lemma}

\newtheorem{remark}{Remark}
\newtheorem{definition}{Definition}
\newtheorem{corollary}{Corollary}
\newtheorem{perspective}{Perspective}

\theoremstyle{definition}

\hypersetup{
    breaklinks=true,   
    pdfusetitle=true,  
    colorlinks=true, 
    linkcolor=green3, 
    citecolor=green3, 
    urlcolor=blue3 
}

\DefineBibliographyStrings{english}{%
  backrefpage = {page},
  backrefpages = {pages},
}

\newcommand{\defeq}{:=}

\newcommand{\norm}[1]{\left\lVert#1\right\rVert}

\newcommand{\normop}[1]{\left\lVert#1\right\rVert_{\textup{op}}}

\newcommand{\normsop}[1]{\lVert#1\rVert_{\textup{op}}}

\newcommand{\inprod}[2]{\left\langle#1, #2\right\rangle}
\newcommand{\inprods}[2]{\langle#1, #2\rangle}
\newcommand{\eps}{\varepsilon}
\renewcommand{\epsilon}{\varepsilon}

\newcommand{\sig}{\sigma}

\newcommand{\R}{\mathbb{R}}

\newcommand{\N}{\mathbb{N}}

\newcommand{\half}{\frac{1}{2}}

\newcommand{\E}{\mathbb{E}}

\newcommand{\Var}{\textup{Var}}

\newcommand{\Nor}{\mathcal{N}}

\newcommand{\id}{\mathbf{I}}

\newcommand{\dd}{\textup{d}}
\usepackage{xcolor}
\definecolor{burntorange}{rgb}{0.8, 0.33, 0.0}

\newcommand{\Par}[1]{\left(#1\right)}
\newcommand{\Brack}[1]{\left[#1\right]}
\newcommand{\Brace}[1]{\left\{#1\right\}}

\newcommand{\Cov}{\mathbf{Cov}}

\newcommand{\mc}[1]{\mathcal{#1}}
\newcommand{\mmc}{\mathbf{C}}

\newcommand{\0}{\mathbf{0}}

\newcommand{\va}{\mathbf{a}}
\newcommand{\vb}{\mathbf{b}}
\newcommand{\vc}{\mathbf{c}}

\newcommand{\vg}{\mathbf{g}}
\newcommand{\vh}{\mathbf{h}}

\newcommand{\vm}{\mathbf{m}}

\newcommand{\vp}{\mathbf{p}}

\newcommand{\vs}{\mathbf{s}}
\newcommand{\vt}{\mathbf{t}}
\newcommand{\vu}{\mathbf{u}}
\newcommand{\vv}{\mathbf{v}}
\newcommand{\vw}{\mathbf{w}}
\newcommand{\vx}{\mathbf{x}}
\newcommand{\vy}{\mathbf{y}}
\newcommand{\vz}{\mathbf{z}}
\newcommand{\vB}{\bm{B}}
\newcommand{\vW}{\bm{W}}
\newcommand{\backmu}{\mu^{\leftarrow}}
\newcommand{\sfQ}{\mathsf{Q}}

\newcommand{\calC}{\mathcal{C}}

\newcommand{\calF}{\mathcal{F}}

\newcommand{\calK}{\mathcal{K}}

\newcommand{\calP}{\mathcal{P}}

\newcommand{\calT}{\mathcal{T}}
\newcommand{\calU}{\mathcal{U}}

\newcommand{\backx}{x^{\leftarrow}}
\newcommand{\backvx}{\vx^{\leftarrow}}
\newcommand{\backpi}{\pi^{\leftarrow}}
\newcommand\mmid{\mathbin{\|}}

\DeclarePairedDelimiterX{\infdivx}[2]{(}{)}{%
  #1\;\delimsize\|\;#2%
}
\newcommand{\KL}{\mathrm{KL}\infdivx}
\newcommand{\D}{\mathrm{d}}
\newcommand{\tilt}{\mathcal{T}}

\newcommand{\msig}{\boldsymbol{\Sigma}}

\newcommand{\mbf}[1]{\mathbf{#1}}
\newcommand\bs[1]{\boldsymbol{#1}}
\newcommand\msf[1]{\mathsf{#1}}
\newcommand{\deq}{\coloneqq}
\newcommand{\sfP}{\mathsf{P}}
\newcommand{\sfL}{\mathsf{L}}
\newcommand{\sfGam}{\mathsf{\Gamma}}

\usepackage{mathrsfs}

\newcommand\msc[1]{\mathscr{#1}}
\newcommand\wt[1]{\widetilde{#1}}

\newcommand{\Ent}{\textup{Ent}}

%% file: sections/intro.tex
\section{Introduction}\label{sec:intro}

Stochastic localization, an elegant stochastic process introduced by \cite{Eldan13}, has proven to be a particularly useful tool in the analysis of high-dimensional distributions, and in designing algorithms to sample from them. Use of this process has notably resulted in a series of gradual improvements to the estimate of the KLS constant \cite{Eldan13, LeeV17, Chen21, KlartagL22, JambulapatiLV22, Klartag23}, i.e., the smallest isoperimetric constant among all isotropic log-concave densities on $\R^d$ \cite{KannanLS95}. It has also enabled various other breakthroughs in probability theory and theoretical computer science \cite{EldanL14, Eldan18, Klartag18, LeeV18, Eldan20, EldanMZ20, Klartag2021Spectral, EldanS22, EldanKZ22, AlaouiMS22, AnariHLVXY23, Guan24, KlartagL24a}.

Among the many useful properties of stochastic localization, the most pertinent to us is the realization that the process is in fact \emph{equivalent} to various other constructions that have arisen from probability theory, theoretical computer science, information theory, and machine learning. This is not a new observation; many results have connected two or more equivalent forms of stochastic localization to obtain novel characterizations or tighter estimates for at least one of the forms~\cite{Lehec13, EldanL15, LeeST21, ChenCSW22, AlaouiM22, ChenE22, BentonBDD24, MikulincerS24, KlartagO25}.

This survey's goal is to present a relatively complete collection of different perspectives on stochastic localization, each equating the base process (cf.\ Perspective~\ref{perspective:tilt}) with another natural probabilistic object. While the proofs presented here are not new (except for providing some missing details), we believe this exposition benefits the community for at least the following reasons. 

\begin{enumerate}[label=(\arabic*)]
\item Although prior expositions of stochastic localization \cite{Klartag2021Spectral, Eldan22, Montanari23, KlartagL24b, Chewi25} have presented several forms of the process, none make it a goal to connect all existing perspectives. Our presentation includes all perspectives contained in these prior works, and our explicit focus is a streamlined presentation of their connections.
\item Several perspectives may be more familiar to a subcommunity interacting with stochastic localization in a particular way. Our presentation simplifies translation between different perspectives by clarifying the connections, which may enable further applications.
\item We make an effort to review the relevant background and keep our exposition self-contained, without being overly verbose. While some derivations are folklore to subcommunities where such calculations are routine, we believe there is value in providing explicit derivations for unfamiliar readers, particularly in an introductory survey.
\end{enumerate}

At this juncture, it is useful to introduce our first perspective on stochastic localization, based on its definition in \cite{LeeV17}, a small modification of its original definition in \cite{Eldan13}.\footnote{More generally, the stochastic localization process of \cite{Eldan13} allows specifying \emph{control matrices} in the dynamics, which affect the covariance of the Gaussian regularization in $\pi_t$. We recall this more general form in Appendix~\ref{app:anisotropic}, focusing on the isotropic variant throughout the main body for simplicity.}

\begin{tcolorbox}[colback=blue!10, colframe=blue!50!black, boxrule=0.5pt, arc=2mm]
\begin{perspective}\label{perspective:tilt}
Let $\pi_0 \in \calP(\R^d)$ be such that $\vm_0 \defeq \E_{\vx \sim \pi_0}[\vx]$ exists. Define a stochastic process $\{\vc_t \in \R^d\}_{t \ge 0}$  as follows, where $\{\vW_t \in \R^d\}_{t \ge 0}$ is a Wiener process:
\begin{equation}\label{eq:tilt_def}\tag{$\msf{SL}$-$\msf{I}$}
\begin{gathered}\vc_0 = \0_d,\quad \dd \vc_t = \vm_t \, \dd t + \dd \vW_t,\\
\text{ where } \vm_t \defeq \E_{\vx \sim \pi_t}[\vx],\; \pi_t(\vx) \propto \exp\Par{\inprod{\vc_t}{\vx} - \frac t 2 \norm{\vx}_2^2}\pi_0(\vx),\text{ for all } t \ge 0.\end{gathered}
\end{equation}
We call the (random) induced measures $\{\pi_t\}_{t \ge 0}$ the \emph{stochastic localization} of $\pi_0$.
\end{perspective}
\end{tcolorbox}

In other words, $\pi_t$ is an \emph{exponential tilt} of the ``Gaussian-regularized'' measure $\propto \exp(-\frac t 2 \norm{\vx}_2^2)\pi_0(\vx)$, by the log-linear function $\exp(\inprod{\vc_t}{\vx})$. The dynamics of the random tilt $\vc_t$ are governed by \eqref{eq:tilt_def}, which includes a bias towards the mean $\vm_t$ of the current measure $\pi_t$.

Why is Perspective~\ref{perspective:tilt} useful? For one thing, it replaces $\pi_0$ with a distribution over measures $\pi_t$, each of which regularizes $\pi_0$ by a randomly-shifted Gaussian. These Gaussians always have variance $\frac 1 t$ at time $t$, and hence as $t \to \infty$, the measure $\pi_t$ obtains strong concentration properties, potentially much stronger than those of $\pi_0$. For this reason, \eqref{eq:tilt_def} is termed a \emph{localization process} (Definition 3, \cite{ChenE22}), in that $\pi_t$ successively ``localizes'' towards a Dirac delta distribution.

In addition to always inducing a fixed amount of Gaussian regularization by time $t$,  \eqref{eq:tilt_def} satisfies a second important property: it is a \emph{measure-valued martingale}, i.e., $\E[\pi_t(\vx)] = \pi_0(\vx)$ pointwise over $\vx \in \R^d$, and for all $t \ge 0$. These two facts let us view the process $\pi_0 \to \pi_t$ as a random decomposition of $\pi_0$ into better-behaved components, whose properties are inherited by $\pi_0$ ``on average.'' Many applications of stochastic localization build on precisely this perspective.

We now outline the remaining perspectives found in this survey.

\begin{enumerate}[label=(\arabic*)]
    \item \textbf{Section~\ref{sec:measure}} formalizes the martingale property of  \eqref{eq:tilt_def}, using a dual perspective (observed in \cite{Eldan13}) of \eqref{eq:tilt_def} as a stochastic process on the measure $\pi_t$ itself, rather than the tilt $\vc_t$.
    \item \textbf{Section~\ref{sec:posterior}} presents an alternative information-theoretic perspective (due to \cite{AlaouiM22}) on the tilt dynamics \eqref{eq:tilt_def} as \emph{posterior sampling} from a noisy Gaussian channel.
    \item \textbf{Section~\ref{sec:diffusion}} presents a perspective that equates a time-changed variant of \eqref{eq:tilt_def} with \emph{denoising diffusion probabilistic models} \cite{HoJA20, SongSKKEP21}, a dominant paradigm in modern diffusion modeling. This connection was observed by \cite{Klartag2021Spectral, Montanari23}.
    \item \textbf{Section~\ref{sec:renormalization}} presents a perspective that uses a continuous renormalization procedure --- Gaussian integration and the Polchinski equation --- to decompose a measure in order to derive functional inequalities; \cite{BauerschmidtBD24} connects this explicitly to stochastic localization.
    \item \textbf{Section~\ref{sec:schrodinger}} presents a perspective via the classical (static) Schr\"odinger bridge problem, which gives another derivation of the Polchinski semigroup in Section~\ref{sec:renormalization}.
    \item \textbf{Section~\ref{sec:dynamic}} and \textbf{Section~\ref{sec:eot}} present two alternative perspectives on the static Schr\"odinger bridge problem, respectively connecting it to a dynamic reformulation based on Girsanov's theorem, and an entropy-regularized optimal transport problem.
\end{enumerate}

Each section begins by introducing a new probabilistic object of self-contained interest, along with the relevant background needed to understand it. It then shows how the newly-introduced object is equivalent to the process in Perspective~\ref{perspective:tilt}, after appropriate transformations. 

Finally, while this survey focuses on the stochastic localization process itself, a related topic of interest is \emph{algorithmic applications} of stochastic localization. There are natural sampling algorithms derived from stochastic localization, several of which have enabled improved runtimes for important applications in theoretical computer science, statistical physics, and machine learning \cite{LeeV18, LeeST21, ChenCSW22, ChenE22, EldanKZ22, AlaouiMS22, AnariHLVXY23, BentonBDD24}. To this end, in Appendix~\ref{app:rgo}, we provide an exposition on how applying the \emph{localization schemes} framework of \cite{ChenE22} to the process \eqref{eq:tilt_def} yields the \emph{restricted Gaussian dynamics} Markov chain \cite{LeeST21}. We also show how techniques developed in the main body provide a mixing time estimate for this Markov chain.

\begin{remark}[On solutions to \eqref{eq:tilt_def}] 
A sufficient condition for the existence and uniqueness of a solution to the process $\vc_t$ is uniform Lipschitz-continuity of $\vm_t$.  A straightforward calculation shows $\nabla_{\vc_t}\vm_t=\nabla_{\vc_t}\mathbb{E}_{\vx\sim \pi_t}[\vx]=\Cov(\pi_t)$; thus, it suffices that $\norm{\Cov(\pi_t)}_{\operatorname{op}}$ is uniformly bounded.  If $\pi_t$ is compactly supported then this is clear.  More generally, if $\pi_t$ satisfies a Poincar\'e inequality \eqref{eq:pi_def}, i.e., 
\[\operatorname{Var}_{\pi_t}[f]\leq \frac{1}{\alpha}\mathbb{E}_{\pi_t}\left[\norm{\nabla f}^2_2\right]\]
for suitable functions $f$, then plugging in linear functions $f(\vx)=\inprods{\vv}{\vx}$ we deduce that $\normsop{\Cov(\pi_t)} \le \frac 1 \alpha$.  If $\pi$ itself is $\alpha$-strongly log-concave, then the arguments of Section~\ref{sec:perspective-adjointheat} show that for all finite $t$, $\pi_t$ satisfies a (weakly) improved Poincar\'e inequality.  We caution that for more general usage, solution existence and uniqueness for \eqref{eq:tilt_def} do not necessarily hold, so the reader should carefully check in the case of their application before applying results in this survey.
\end{remark}

\subsection{Notation}

For $n \in \N$ we denote $[n] \defeq \{i \in \N \mid i \le n\}$. Vectors are denoted in upright boldface lowercase, and matrices are denoted in upright boldface uppercase, unless specified otherwise. We reserve use of italic boldface uppercase letters, e.g., $\{\vW_t\}_{t \ge 0}$, $\{\vB_t\}_{t \ge 0}$ for Wiener processes in $\R^d$. 
All probability measures are given as densities relative to the Lebesgue measure over $\R^d$, and all integrals are over $\R^d$ unless specified otherwise. We use $\calP(\R^d)$ to denote the set of probability measures over $\R^d$. When $f: \R^d \to \R_{\ge 0}$ is integrable, we use $\pi \propto f$ to mean that $\pi$ is the measure that equals $f$ up to a normalization constant $(\int f(\vx) \dd \vx)^{-1}$. We use $\Nor(\vm, \msig)$ to denote the multivariate normal distribution with specified mean and covariance, $\vm \in \R^d$ and $\msig \in \R^{d \times d}$. We use $\id_d$ to denote the identity matrix in $\R^d$, and $\0_d$ to denote the all-zeroes vector in $\R^d$. We use $\nabla$, $\nabla^2$, $\nabla \cdot$, $\Delta$ to denote the gradient, Hessian, divergence, and Laplacian respectively. When $f$ is a function depending on a variable $t$, we often use $\partial_t f$ as a shorthand for $\frac {\partial f} {\partial t}$.

In Sections~\ref{sec:posterior},~\ref{sec:schrodinger}, and~\ref{sec:eot}, we require tools specialized to \emph{path measures}, i.e., probability measures supported on continuous paths on $\R^d$ indexed by a time in $[0, t]$. We denote this support by $\calC([0, t] \times \R^d)$, and the space of associated path measures by $\calP([0, t] \times \R^d)$.  We always denote path measures with capital letters, e.g., $P \in \calP([0, t] \times \R^d)$ is supported on continuous paths $\vp_{[0, t]} \in \calC([0, t] \times \R^d)$. Hence, $\vp_s \in \R^d$ for any $s \in [0, t]$, and $P_s$ denotes the law of $\vp_s$.

\subsection{Preliminaries}\label{ssec:prelims}

We repeatedly use the following two standard facts in stochastic calculus.

\begin{lemma}[It\^o's lemma]\label{lem:ito}
	Let $f: \R^d \to \R$ be twice-differentiable, and suppose $\{\vx_t\}_{t \ge 0}$ follows the SDE $\dd \vx_t = \vm_t \, \dd t + \msig_t \, \dd \vW_t$ where $\{\vm_t, \msig_t\}_{t \ge 0}$ are adapted to the filtration generated by a Wiener process $\{\vW_t\}_{t \ge 0}$. Then $\{f(\vx_t)\}_{t \ge 0}$ is also a drift-diffusion process, following the SDE
	\begin{align*}\dd f(\vx_t) &= \inprod{\nabla f(\vx_t)}{\dd \vx_t} + \half\inprod{\nabla^2 f(\vx_t)}{\msig_t\msig_t^\top}\, \dd t   \\
		&= \Bigl(\inprod{\nabla f(\vx_t)}{\vm_t} + \half\inprod{\nabla^2 f(\vx_t)}{\msig_t\msig_t^\top}\Bigr) \,\dd t + \inprod{\nabla f(\vx_t)}{\msig_t \, \dd \vW_t}.\end{align*}
\end{lemma}

\begin{lemma}[Fokker-Planck equation]\label{lem:fokker_planck}
Suppose $\{\vx_t\}_{t \ge 0}$ follows the SDE $\dd \vx_t = \vm_t(\vx_t) \, \dd t + \msig_t(\vx_t)\,\dd \vW_t$ where $\{\vm_t, \msig_t\}_{t \ge 0}$ are adapted to the filtration generated by a Wiener process $\{\vW_t\}_{t \ge 0}$. Then letting $\pi_t$ denote the law of $\vx_t$, we have for all $\vx \in \R^d$ that
\[\partial_t \pi_t(\vx) = -\nabla \cdot\Par{\vm_t(\vx) \pi_t(\vx)} + \half \sum_{(i, j) \in [d] \times [d]} \partial_{\vx_i}\partial_{\vx_j}\Par{\msig_t(\vx)\msig_t(\vx)^\top \pi_t(\vx)}_{ij}.\]
\end{lemma}

In Sections~\ref{sec:renormalization} and~\ref{sec:schrodinger}, and Appendix~\ref{app:rgo}, we use tools from Markov semigroup theory. A (time-inhomogeneous) Markov semigroup $\{\sfP_{\sig,\tau}\}_{0 \le \sig \le \tau}$ is induced by a stochastic process $\{\vx_\tau\}_{\tau \ge 0}$ via the following definition: for a 
compactly supported test function $f: \R^d \to \R$, we let 
\[\sfP_{\sig,\tau} f(\vx) \defeq \E[f(\vx_\tau) \mid \vx_\sig = \vx].\]
Note that by iterating expectations, the semigroup property $\sfP_{\rho, \sig} \sfP_{\sig, \tau} = \sfP_{\rho, \tau}$ holds.\footnote{$\{\calP_{\sig, \tau}\}_{0 \le \sig \le \tau \le 1}$ is not a semigroup in the standard mathematical sense, as elements can only be composed if they share an index. We follow the terminology used by the literature on time-inhomogeneous stochastic processes.} 

The adjoint operator $\sfP_{\sig,\tau}^*$ is then interpreted in duality as $(\msf P_{\sigma, \tau}^*\delta_{\mathbf x})(f) = \sfP_{\sig,\tau} f(\vx)$. This means that applying $\sfP_{\sig, \tau}^*$ to a density over $\vx$ at time $\sig$ advances the density to that of $\vx$ at time $\tau$. 

We denote the \emph{infinitesimal generator} of the semigroup at time $\tau$ by $\sfL_\tau$, which operates as:
\begin{equation}\label{eq:generator_tau}\sfL_\tau f \defeq \lim_{\eta \searrow 0} \frac{\sfP_{\tau, \tau + \eta} f - f}{\eta}.\end{equation}

Finally, we define the \emph{carr\'e du champ} operator at time $\tau$ as:
\begin{equation}\label{eq:cdc}
\begin{aligned}
    \msf{\Gamma}_\tau(f, g)(\vx) 
    &\defeq \half \lim_{\eta \searrow 0} \frac{\E\Brack{(f(\vx_{\tau + \eta}) - f(\vx_\tau))(g(\vx_{\tau + \eta}) - g(\vx_\tau)) \mid \vx_\tau = \vx}}{\eta}\\
    &=\frac{1}{2}\Par{\sfL_\tau(fg)(\vx)-f(\vx)\sfL_\tau g(\vx)-g(\vx)\sfL_\tau f(\vx)}.
\end{aligned}
\end{equation}

%% file: sections/measure.tex
\section{Measure-valued process}\label{sec:measure}

We next present a dual perspective on the process \eqref{eq:tilt_def}: rather than track how the random tilt $\vc_t$ evolves, we track the dynamics of its induced distribution $\pi_t(\vx) \propto \exp(\inprod{\vc_t}{\vx} - \frac t 2 \norm{\vx}_2^2) \pi_0(\vx)$.

\begin{tcolorbox}[colback=blue!10, colframe=blue!50!black, boxrule=0.5pt, arc=2mm]
\begin{perspective}\label{perspective:density}
Let $\pi_0\in \calP(\R^d)$ be such that $\vm_0 \defeq \E_{\vx \sim \pi_0}[\vx]$ exists. Define a measure-valued stochastic process $\{\pi_t\}_{t \ge 0}$ as follows, where $\{\vW_t\}_{t \ge 0}$ is a Wiener process:
\begin{equation}\label{eq:density_def}\tag{$\msf{SL}$-$\msf{II}$}
\begin{gathered}
    \dd \pi_t(\vx) = \inprod{\vx - \vm_t}{\dd \vW_t}\pi_t(\vx),\\\text{ pointwise over } \vx \in \R^d,\text{ where } \vm_t \defeq \E_{\vx \sim \pi_t}[\vx].
\end{gathered}
\end{equation}
\end{perspective}
\end{tcolorbox}

We first make a simple observation regarding correctness of the dynamics \eqref{eq:density_def}.

\begin{lemma}\label{lem:pit_density}
In \eqref{eq:density_def}, $\int \pi_t(\vx) \, \dd \vx = 1$ for all $t \ge 0$, and $\E[\pi_t(\vx)] = \pi_0(\vx)$ for all $\vx \in \R^d$ and $t \ge 0$.
\end{lemma}
\begin{proof}
The first part asks to show that 
\[\dd\Par{ \int \pi_t(\vx) \dd \vx} = 0,\]
from which we can conclude $\int \pi_t(\vx) \dd \vx = 1$ for all $t \ge 0$. To see this, we apply \eqref{eq:density_def}:
\[\dd\Par{ \int \pi_t(\vx) \, \dd \vx} = \int \inprod{\vx - \vm_t}{\dd \vW_t} \pi_t(\vx)\, \dd \vx = \inprod{\vm_t - \vm_t}{\dd \vW_t} = 0.\]
The second part is immediate because $\pi_t(\vx)$ is a martingale in \eqref{eq:density_def} for all $\vx \in \R^d$.
\end{proof}

Thus, $\pi_t$ always remains a valid probability measure, regardless of the realization of $\{\vW_t\}_{t \ge 0}$.
We now state the main result of this section: the following equivalence, shown by \cite{Eldan13}.

\begin{theorem}\label{thm:tilt_density}
The dynamics of $\pi_t$ given by \eqref{eq:tilt_def} and \eqref{eq:density_def} are the same.
\end{theorem}
\begin{proof}
Let us start with the dynamics \eqref{eq:density_def}. By It\^o's lemma (Lemma~\ref{lem:ito}),
\begin{equation}\label{eq:log_sde}\dd \log\pi_t(\vx) = \inprod{\vx - \vm_t}{\dd \vW_t} - \half\norm{\vx - \vm_t}_2^2\,  \dd t = \inprod{\vx}{\dd \vW_t + \vm_t \dd t} - \half\norm{\vx}_2^2 \, \dd t + C_t,\end{equation}
where $C_t = -\inprod{\vm_t}{\dd \vW_t} - \half\norm{\vm_t}_2^2 \, \dd t$ is independent of $\vx$. We know from Lemma~\ref{lem:pit_density} that $C_t$ must be chosen so that $\pi_t$ stays a density, so by integrating \eqref{eq:log_sde}, we conclude that \eqref{eq:tilt_def} holds:
\[\pi_t(\vx) \propto \exp\Par{\inprod{\vc_t}{\vx} - \frac t 2 \norm{\vx}_2^2} \pi_0(\vx),\quad \dd \vc_t = \dd \vW_t + \vm_t\, \dd t. \qedhere\]
\end{proof}

Combining Lemma~\ref{lem:pit_density} and Theorem~\ref{thm:tilt_density} shows our earlier claim from Section~\ref{sec:intro}: that \eqref{eq:tilt_def} induces a measure-valued martingale decomposing $\pi_0$ into randomly-tilted measures $\pi_t$. 

It is an instructive exercise to reverse Theorem~\ref{thm:tilt_density}, i.e., start from the dynamics \eqref{eq:tilt_def} and derive the equivalence to~\eqref{eq:density_def}. This follows by rewriting $\dd \log \pi_t(\vx) = \inprod{\dd \vc_t}{\vx} - \half \norm{\vx}_2^2 - \dd \log(Z_t)$, where $Z_t \defeq \int \exp(\inprod{\vc_t}{\vx} - \frac t 2 \norm{\vx}_2^2) \pi_0(\vx) \, \dd \vx$, and using It\^o's lemma to show
\[\dd \log(Z_t) = \inprod{\dd \vc_t}{\vm_t} - \half \norm{\vm_t}_2^2 \, \dd t.\]
This gives the same form of $\dd \log \pi_t(\vx)$ as computed earlier in \eqref{eq:log_sde}, and indeed, applying It\^o's lemma once more completes the derivation of $\dd \pi_t(\vx)$ as in Perspective~\ref{perspective:density}. We provide this calculation in Appendix~\ref{app:anisotropic}, for more general anisotropic stochastic localization processes.

%% file: sections/posterior.tex
\section{Posterior estimation}\label{sec:posterior}

In this section, we present an information-theoretic perspective on stochastic localization due to \cite{AlaouiM22}. This perspective recasts the tilt $\vc_t$ as a noisy observation from a Gaussian channel.

\begin{tcolorbox}[colback=blue!10, colframe=blue!50!black, boxrule=0.5pt, arc=2mm]
\begin{perspective}\label{perspective:posterior}
Let $\pi_0 \in \calP(\R^d)$ be such that $\vm_0 \defeq \E_{\vx \sim \pi_0}[\vx]$ exists, and let $\vx \sim \pi_0$. Define noisy observations $\{\vc_t \defeq t\vx + \vB_t\}_{t \ge 0}$, where $\{\vB_t\}_{t \ge 0}$ is a Wiener process.
\end{perspective}
\end{tcolorbox}

Conditioned on $\vx \sim \pi_0$, the output of the noisy channel $\vc_t$ follows $\vc_t \sim \Nor(t\vx, t \id_d)$. This implies that the information content of $\vc_t$ is increasing, in the sense that $\frac 1 t \vc_t = \Nor(\vx, \frac 1 t \id_d)$ becomes a better estimate of $\vx$ as $t \to \infty$, which again captures the ``localization'' behavior of the process.

Before introducing the main result of this section, we require two helper facts. The first is a reparameterization of a stochastic process in terms of conditional means.

\begin{lemma}[{\cite[Theorem 7.12]{LipsterS01}}]\label{lem:mean_sde}
Let $\dd \vc_t = \vv_t \, \dd t + \dd \vB_t$, $\vc_0 = \0_d$, where $\{\vB_t\}_{t \ge 0}$ is a Wiener process. Let $\vm_t \defeq \E[\vv_t \mid \mathscr{F}_t]$ for all $t \ge 0$, where $\{\mathscr{F}_t\}_{t \ge 0}$ is the filtration generated by $\{\vc_t\}_{t \ge 0}$. Then we also have that $\dd \vc_t = \vm_t \, \dd t + \dd \vW_t$ where $\{\vW_t\}_{t \ge 0}$ is a Wiener process. 
\end{lemma}

We omit the proof of Lemma~\ref{lem:mean_sde}, but mention here that it follows by using It\^o's lemma to compute the characteristic function of $\vc_t - \int_0^t \vm_s \, \dd s$, which is the same as that of a Wiener process. 

The second helper fact is the Cameron-Martin theorem (a specialization of the more general Girsanov's theorem, cf.\ Lemma~\ref{lem:girsanov}), derived by similar means as those used to show Lemma~\ref{lem:mean_sde}. It characterizes how adding a drift $\vv_s \,\dd s$ changes the \emph{path measure} of a Wiener process.

\begin{lemma}[{Cameron-Martin theorem; \cite[Theorem 5.24]{Gall16}}]\label{lem:cameron_martin}
Let $\{\vh_s\}_{s \in [0, t]}$ follow the ODE $\vh_0 = \0_d$, $\dd \vh_s = \vv_s \, \dd s$, where $\int_0^t \norm{\vv_s}_2^2 \, \dd s < \infty$, and let $\vc_s = \vh_s + \vW_s$ for all $s \in [0, t]$, where $\{\vW_s\}_{s \ge 0}$ is a Wiener process. 

Let $\vp_{[0, t]} \in \calC([0, t])$ index a continuous path $\{\vp_s\}_{s \in [0, t]}$ in $\R^d$, let $W$ denote the Wiener measure of $\vp_{[0, t]} = \vW_{[0, t]}$ (i.e., the density of $\vW_s = \vp_s$ for all $s \in [0, t]$), and let $P$ denote the path measure of $\vp_{[0, t]} = \vc_{[0, t]}$. Then for all $\vp_{[0, t]} \in \calC([0, t])$, 
\[\frac{P(\vp_{[0, t]})}{W(\vp_{[0, t]})} = \exp\Par{\int_0^t \inprod{\vv_s}{\dd \vp_s} - \half \int_0^t \norm{\vv_s}_2^2 \, \dd s}.\]
\end{lemma}

We are now ready to state the equivalence between Perspectives~\ref{perspective:tilt} and~\ref{perspective:posterior}, due to \cite{AlaouiM22}.

\begin{theorem}\label{thm:tilt_posterior}
The dynamics of $\vc_t$ given by \eqref{eq:tilt_def} and in Perspective~\ref{perspective:posterior} are the same. Moreover, the induced measure $\pi_t$ in \eqref{eq:tilt_def} is the same as the posterior distribution $\pi_0(\vx \mid \vc_t)$ in Perspective~\ref{perspective:posterior}. 
\end{theorem}
\begin{proof}
To see the second claim, regardless of the realization of $\vc_t$, we have that
\[\pi_0(\vx \mid \vc_t) \propto \exp\Par{-\frac 1 {2t}\norm{t\vx - \vc_t}_2^2}\pi_0(\vx) \propto \exp\Par{\inprod{\vc_t}{\vx} - \frac t 2 \norm{\vx}_2^2}\pi_0(\vx),\]
where we used Bayes' theorem in the above derivation. This agrees with the induced $\pi_t$ in \eqref{eq:tilt_def}.

Next, we claim that the distribution of $\vx \mid \mathscr{F}_t$ is the same as $\vx \mid \vc_t$, i.e., specifying just the endpoint $\vc_t$ gives as much information about $\vx$ as the entire path $\{\vc_s\}_{0 \le s \le t}$. Equivalently, $\vc_t$ is a sufficient statistic for $\vx$ with respect to $\mathscr{F}_t$. This follows from Lemma~\ref{lem:cameron_martin} with $\vv_s = \vx$ for all $s \in [0, t]$, which shows that the joint measure for $\vx \sim \pi_0$ and the induced path $\vc_{[0, t]} \in \calC([0, t])$ is
\[\propto \exp\Par{\int_0^t \inprod{\vx}{\dd \vc_s} - \frac t 2 \norm{\vx}_2^2} \pi_0(\vx) 
W(\vc_{[0, t]}) \propto \exp\Par{\inprod{\vc_t}{\vx} - \frac t 2 \norm{\vx}_2^2}\pi_0(\vx) W(\vc_{[0, t]}). \]
Note that $\vx$ only interacts with $\vc_t$ in the above factorization, so sufficiency of $\vc_t$ for $\vx$ follows from the Fisher--Neyman factorization theorem.
We conclude by proving the first claim.  
Starting from $\{\vc_t\}_{t \ge 0}$ in Perspective~\ref{perspective:posterior}, we apply Lemma~\ref{lem:mean_sde} with the substitution $\vv_t \gets \vx$. This shows that 
\begin{equation}\label{eq:y_sde}
\dd \vc_t = \vm_t \, \dd t + \dd \vW_t, \text{ where } \vm_t \defeq \E[\vx \mid \mathscr{F}_t]= \E[\vx \mid \vc_t] = \E_{\pi_t}[\vx]. \qedhere
\end{equation}
\end{proof}

%% file: sections/diffusion.tex
\section{Diffusion models}\label{sec:diffusion}

In this section, we relate stochastic localization to a framework known as \emph{denoising diffusion probablistic models} (DDPMs), which was popularized by \cite{HoJA20, SongSKKEP21} and is widely used in the practice of generative modeling. The goal of diffusion models in the context of generative modeling is to sample from a distribution $\pi$, given a dataset of i.i.d.\ draws from it. 

The DDPM framework in particular is based on a time-changed backwards Ornstein--Uhlenbeck (OU) process. 
Recall that the OU process is the Langevin dynamics with a standard Gaussian as its stationary measure, i.e., it follows the following SDE from $\vx_0 \in \R^d$:
\begin{equation}\label{eq:ou_process}\dd \vx_t = -\vx_t \, \dd t + \sqrt 2 \, \dd \vB_t, \text{ where } \{\vB_t\}_{t \ge 0} \text{ is a Wiener process.}\end{equation}

We make the following simple observation.
\begin{lemma}\label{lem:ou_dist}
The distribution of $\vx_t \mid \vx_0$ in \eqref{eq:ou_process} is $\Nor(\exp(-t) \vx_0, (1 - \exp(-2t)) \id_d)$.
\end{lemma}
\begin{proof}
Solving the SDE in \eqref{eq:ou_process} as an It\^o integral, we have
\[\vx_t = \exp(-t)\vx_0 + \sqrt 2 \int_0^t \exp(-(t - s))\,\dd \vB_s.\]
The mean of $\vx_t$ is thus $\exp(-t) \vx_0$ as claimed, and the covariance scales proportionally to the quadratic variation, i.e.,
$2\int_0^t \exp(-2(t - s))\,\dd s = 1 - \exp(-2t)$.
\end{proof}

Let us define $\pi_t$ to be the distribution of $\vx_t$ according to \eqref{eq:ou_process}, where we draw $\vx_0 \sim \pi_0$, the target of our sampling algorithm. We apply a one-to-one backwards time change $\tau: [0, \infty] \to [0, \infty]$, i.e., a $\tau$ that satisfies $\tau'(t) < 0$ for all times $t \ge 0$, and sets $\tau(0) = \infty$ and $\tau(\infty) = 0$. This induces an equivalent backwards process, indexed by a backwards time $u \ge 0$:
\begin{equation}\label{eq:backdef}\backvx_{u} \defeq \vx_{\tau^{-1}(u)},\quad \backpi_u \defeq \pi_{\tau^{-1}(u)}. \end{equation}
To see how \eqref{eq:backdef} relates to generative models, if we let $\backvx_0 = \vx_\infty$, which we can simulate with a draw from the stationary distribution $\Nor(\0_d, \id_d)$, then running the ``backwards process'' and producing a sample  $\backvx_\infty$ is equivalent to drawing from the target distribution $\backpi_\infty = \pi_0$.

We now give a helpful characterization of the backwards process $\{\backvx_u\}_{u \ge 0}$. The following result is standard, and an early derivation of it can be found, e.g., in \cite{Anderson82}.

\begin{lemma}\label{lem:timechange}
	Let $\dd \vx_t = \vx_t \, \dd t + \sqrt{2} \, \dd \vB_t$, where $\{\vB_t\}_{t \ge 0}$ is a Wiener process. Let $\tau: [0, \infty] \to [0, \infty]$ satisfy $\tau'(t) < 0$ for all $t \ge 0$, $\tau(0) = \infty$, $\tau(\infty) = 0$. Letting $\pi_t$ be the law of $\vx_t$, and following the notation \eqref{eq:backdef}, we have for a Wiener process $\{\vW_u\}_{u \ge 0}$ that
	\[\dd \backvx_u = |(\tau^{-1})'(u)| \bigl(\backvx_u + 2\nabla \log \backpi_u(\backx_u)\bigr) \,  \dd u + \sqrt{2|(\tau^{-1})'(u)|}\,  \dd \vW_u\,. \]
\end{lemma}
\begin{proof}
	By the Fokker-Planck equation (Lemma~\ref{lem:fokker_planck}), the measures $\{\pi_t\}_{t \ge 0}$ follow the PDE:
	\begin{align*}\partial_t \pi_t (\vx) &= -\nabla \cdot \Par{-\vx\pi_t(\vx)} + \Delta \pi_t(\vx) \\
		&= \nabla \cdot \bigl(\Par{\vx + \nabla \log \pi_t(\vx)}\pi_t(\vx)\bigr) \\
		&= \nabla \cdot \bigl(\Par{\vx + 2\nabla\log \pi_t(\vx)} \pi_t(\vx)\bigr) - \Delta\pi_t(\vx)\,.
	\end{align*}
    Performing a change of variables, we have
    \begin{align*}\partial_u \backpi_{u}(\backvx) &= -|(\tau^{-1})'(u)|\nabla \cdot\Par{(\backvx + 2\nabla \log \backpi_u(\backvx))\backpi_u(\backvx)} + |(\tau^{-1})'(u)| \Delta \backpi_u(\backvx)\,.\end{align*}
    The conclusion follows by inverting the Fokker-Planck equation with respect to the above display.
\end{proof}

We can now introduce our main object of study in this section.

\begin{tcolorbox}[colback=blue!10, colframe=blue!50!black, boxrule=0.5pt, arc=2mm]
\begin{perspective}\label{perspective:diffusion}
Let $\pi_0 \in \calP(\R^d)$ be such that $\E_{\vx \sim \pi_0}[\vx]$ exists. Define a stochastic process $\{\backvx_u\}_{u \ge 0}$ as follows,
where $\{\vW_u\}_{u \ge 0}$ is a Wiener process:
\begin{equation}\label{eq:backxdef}\tag{$\msf{rev}$}
\begin{gathered}
\backvx_0 \sim \Nor(\0_d, \id_d),\\ \dd \backvx_u = \Par{\frac{\backvx_u}{2u(u + 1)} + \frac 1 {u(u + 1)} \nabla \log \backpi_u(\backvx_u)}\, \dd u + \frac 1 {\sqrt{u(u + 1)}}\,\dd \vW_u.
\end{gathered}
\end{equation}
\end{perspective}
\end{tcolorbox}

Note that \eqref{eq:backxdef} is simply the SDE driving the backwards dynamics \eqref{eq:backdef}, according to our derivation in Lemma~\ref{lem:timechange}, under the time change
\begin{equation}\label{eq:our_tau}
\tau(t) = \frac 1 {\exp(2t) - 1},\quad \tau^{-1}(u) = \half\log\Par{\frac{u + 1}{u}},\quad (\tau^{-1})'(u) = -\frac 1 {2u(u + 1)}.
\end{equation}
We remark that \eqref{eq:backxdef} is not the original backwards SDE from \cite{SongSKKEP21}, Eq.\ (6), which reads
\begin{equation}\label{eq:ddpm_orig}\dd \backvx_u = \Par{\backvx_u + 2\nabla \log\backpi_u(\backvx_u)} \, \dd u + \sqrt 2 \, \dd \vW_u\,.\end{equation}
The above SDE is the result of applying Lemma~\ref{lem:timechange} with a time change satisfying $(\tau^{-1})'(u) = -u$. However, because there is no well-defined $\tau: [0, \infty] \to [0, \infty]$ with this property, applications in finite time require cutting off the forwards SDE at a time $T > 0$ such that $\pi_T \approx \pi_\infty = \Nor(\0_d, \id_d)$, and only defining the backwards SDE up to time $T$ (i.e., $\tau(t) = T - t$). To avoid these boundary issues, Perspective~\ref{perspective:diffusion} uses  \eqref{eq:our_tau}, which just gives a reparameterization of \eqref{eq:ddpm_orig} if we restrict to $[0, T]$. 

As a final technical tool, we require Tweedie's formula (cf.\ \cite{Robbins56}) to rewrite the $\nabla \log \backpi$ term, which is often called the \emph{score} of the backwards process.

\begin{lemma}[Tweedie's formula]\label{lem:tweedie}
	Let $\vy \sim \Nor(\vx, \sigma^2 \id_d)$ be the output of $\vx \sim \pi$ passed through a noisy Gaussian channel, let $\pi(\cdot \mid \vy)$ be the posterior distribution, and let $\nu(\vy)$ be the marginal of $\vy$:
	\[\pi(\vx \mid \vy) \propto \exp\Bigl(-\frac 1 {2\sigma^2}\norm{\vx - \vy}_2^2\Bigr) \pi(\vx)\,,\quad \nu(\vy) \propto \int \exp\Bigl(-\frac 1 {2\sigma^2}\norm{\vx - \vy}_2^2\Bigr)\pi(\vx)\, \dd \vx.\]
	Then, 
	$\nabla \log \nu(\vy) = \frac 1 {\sigma^2}(\E_{\pi(\cdot \mid \vy)}[\vx] - \vy)$.
\end{lemma}
\begin{proof}
	Note that the proportionality constant in $\nu$ does not affect $\nabla \log \nu(\vy)$. Hence,
	\begin{align*}
		\nabla \log \nu(\vy) &= \frac 1 {\int \exp(-\frac 1 {2\sigma^2}\norm{\vx - \vy}_2^2)\pi(\vx)\dd \vx}\int \frac {\vx - \vy}{\sigma^2}\exp\Bigl(-\frac 1 {2\sigma^2}\norm{\vx - \vy}_2^2\Bigr) \pi(\vx) \, \dd \vx \\
		&= \frac 1 {\sigma^2}\Par{\E_{\pi(\cdot \mid \vy)}[\vx] - \vy}\,.\qedhere
	\end{align*}
\end{proof}

Lemma~\ref{lem:tweedie} is remarkably useful in practical applications. Recall from Lemma~\ref{lem:ou_dist} that intermediate distributions in  the forwards and backwards processes are outputs of noisy Gaussian channels, initialized from $\pi$. Further, the backwards SDEs \eqref{eq:backxdef}, \eqref{eq:ddpm_orig} are entirely explicit, except for the score term $\nabla \log \backpi_u$.
Lemma~\ref{lem:tweedie} says that we can approximate this score term by predicting appropriate posterior means $\E_{\pi(\cdot \mid \vy)}$, where $\vy = \backvx_u$ is an intermediate iterate. In practice, this predictor can be learned by minimizing an empirical risk over samples from $\pi$, and modern machine learning models (e.g., deep neural networks) achieve good prediction error. Moreover, fairly strong bounds are available that convert errors arising from score approximation and finite-time discretization to errors in the sampling process, see e.g., a line of work initiated by \cite{ChenCLLSZ23, LeeLT23}.

We conclude by relating Perspective~\ref{perspective:diffusion} to Perspective~\ref{perspective:tilt}, as observed by \cite{Klartag2021Spectral, Montanari23}. 

\begin{theorem}\label{thm:tilt_ddpm}
The processes $\{\vc_u\}_{u \ge 0}$ and $\{\backvx_u\}_{u \ge 0}$ in \eqref{eq:tilt_def}, \eqref{eq:backxdef}, satisfy $\sqrt{u(u + 1)} \backvx_u = \vc_u$.
\end{theorem}
\begin{proof}
First, consider the SDE \eqref{eq:tilt_def}. Performing the change of variables $\backvx_u = \frac 1{\sqrt{u(u + 1)}} \vc_u$,
\begin{equation}\label{eq:rescaled_sde}
\begin{aligned}
\dd \backvx_u &= \Par{\frac{\dd}{\dd u} \frac{1}{\sqrt{u(u + 1)}}} \vc_u \, \dd u + \frac 1 {\sqrt{u(u + 1)}} \, \dd \vc_u \\
&= -\frac{2u + 1}{2(u(u + 1))^{\frac 3 2}} \vc_u \, \dd u + \frac 1 {\sqrt{u(u + 1)}} \vm_u \, \dd u + \frac 1 {\sqrt{u(u + 1)}} \, \dd \vW_u \\
&= -\frac{2u + 1}{2u(u + 1)} \backvx_u \, \dd u + \frac 1 {\sqrt{u(u + 1)}} \vm_u \, \dd u + \frac 1 {\sqrt{u(u + 1)}} \, \dd \vW_u ,
\end{aligned}
\end{equation}
where $\vm_u$ is the mean of the measure $\propto \exp(\vc_u^\top \vx - \frac u 2 \norm{\vx}_2^2) \pi_0(\vx)$, as in \eqref{eq:tilt_def}. 

On the other hand, observe that $\backvx_u$ in \eqref{eq:backxdef} is distributed as $\vx_t$ in \eqref{eq:ou_process}, for $t = \half\log(\frac{u + 1}{u})$. By Lemma~\ref{lem:ou_dist}, we have $\vx_t \sim \Nor(\exp(-t)\vx_0, (1 - \exp(-2t))\id_d)$. Now Tweedie's formula (Lemma~\ref{lem:tweedie}), applied with $\vy \gets \vx_t$, $\vx \gets \exp(-t)\vx_0$, and $\sigma^2 \gets 1 - \exp(-2t)$, gives
\begin{equation}\label{eq:score_calc}
	\begin{aligned}\nabla \log \pi_t(\vx_t) &= \frac 1 {1 - \exp(-2t)}\Par{\E_{\pi_t}[\exp(-t) \vx_0] - \vx_t} \\
		\implies \nabla \log \backpi_{u}(\backvx_u) &=  \sqrt{u(u + 1)}\E_{\backpi_u}\Brack{\vx_0} - (u + 1)\backvx_u .
	\end{aligned}
\end{equation}

Combining \eqref{eq:backxdef} and \eqref{eq:score_calc} yields
\begin{align}\label{eq:x-back-ito}
\begin{aligned}\dd \backvx_u &= \Bigl(\frac{\backvx_u}{2u(u + 1)} + \frac 1 {\sqrt{u(u + 1)}} \E_{\backpi_u}\Brack{\vx_0} - \frac {\backvx_u} u \Bigr) \dd u + \frac 1 {\sqrt{u(u + 1)}} \, \dd \vW_u \\
	&= \Bigl(-\frac{2u + 1}{2u(u + 1)} \backvx_u + \frac 1 {\sqrt{u(u + 1)}} \E_{\backpi_u}\Brack{\vx_0} \Bigr)\, \dd u + \frac 1 {\sqrt{u(u + 1)}} \, \dd \vW_u.
\end{aligned}
\end{align}
To connect \eqref{eq:rescaled_sde} and \eqref{eq:x-back-ito}, recall that $\backpi_u$ is the density of $\Nor(\exp(-t)\vx_0, (1 - \exp(-2t))\id_d) = \Nor(\sqrt{\frac u {u + 1}}\vx_0, \frac 1 {u + 1} \id_d)$, so under $\backpi_u$, the posterior distribution that $\vx_0$ follows (cf.\ Lemma~\ref{lem:tweedie}) is
\[\propto \exp\Par{-\frac{u + 1}{2}\norm{\sqrt{\frac u {u + 1}}\vx_0 - \backvx_u}_2^2}\pi_0(\vx_0) \propto \exp\Par{\vc_u^\top \vx_0 - \frac u 2 \norm{\vx_0}_2^2} \pi_0(\vx_0),\]
which agrees with the distribution used to compute $\vm_u$ in \eqref{eq:rescaled_sde} as desired.
\end{proof}

%% file: sections/renormalization.tex
\section{Renormalization}\label{sec:renormalization}

We present another perspective called renormalization \cite{WilsonK74, Polchinski84}.  The idea is to decompose a measure $\pi_0$ via infinitesimal convolutions, and describe how so-called ``renormalized potentials'' evolve as a Markov process. We follow the recent presentation of \cite{BauerschmidtBD24}.

In this section, we use $\sigma, \tau$ to denote times in $[0, 1]$, and $\calF(\R^d)$ denotes the set of smooth compactly-supported measurable functions on $\R^d$.\footnote{An extended discussion on Markov semigroup theory is beyond our scope. The results here extend beyond the class $\calF(\R^d)$ defined here via approximation; we refer the reader to \cite{BakryGL13} for additional background.} We now introduce our renormalized potentials and measures $\{V_\tau, \nu_\tau\}_{\tau \in [0, 1]}$, and the Polchinski semigroup
$\{\sfP_{\sig, \tau}\}_{0 \le \sig \le \tau \le 1}$ that they induce.

\begin{tcolorbox}[colback=blue!10, colframe=blue!50!black, boxrule=0.5pt, arc=2mm]
\begin{perspective}\label{perspective:renorm}
Let $\pi_0 \in \calP(\R^d)$ be such that $\E_{\vx \sim \pi_0}[\vx]$ exists. Define a family of \emph{renormalized potentials and measures} as follows: let $V_1(\vx) \defeq -\log\pi_0(\vx) - \half\norm{\vx}_2^2$, and
\begin{equation}\label{eq:renorm_def}\tag{$\msf{renorm}$}
\begin{aligned}
V_\tau(\vx) &\defeq -\log\Par{\E_{\vz \sim \Nor(\0_d, (1 - \tau)\id_d)}\Brack{\exp\Par{-V_1(\vx + \vz)}}},\\
\nu_\tau(\vx) &\propto \exp\Par{-V_\tau(\vx) - \frac 1 {2\tau}\norm{\vx}_2^2},\text{ for all } \tau \in [0, 1].
\end{aligned}
\end{equation}
Also, define the induced \emph{Polchinski semigroup}: for all $0 \le \sig \le \tau \le 1$, let
\begin{equation}\label{eq:sg_def}\tag{$\msf{PSG}$}
\sfP_{\sig, \tau} f(\vx) \defeq \exp\Par{V_\sig(\vx)} \E_{\vz \sim \Nor(\0_d, (\tau - \sig)\id_d)}\Brack{\exp\Par{-V_\tau(\vx + \vz)} f(\vx + \vz)}.
\end{equation}
\end{perspective}
\end{tcolorbox}

It is clear that $\nu_1 = \pi_0$ and $\nu_0= \delta_{\0_d}$ is a Dirac measure by weak continuity.

We next clarify the relationship between \eqref{eq:renorm_def} and \eqref{eq:sg_def}. Intuitively, $\sfP_{\sig, \tau}$ advances time from $\sig$ to $\tau$ (in the sense described in Section~\ref{ssec:prelims}), and the corresponding densities are $\nu_\sig$, $\nu_\tau$.

\begin{lemma}\label{lem:semigroup}
Following notation in Perspective~\ref{perspective:renorm}, we have that $\sfP_{\rho, \sig} \sfP_{\sig, \tau} = \sfP_{\rho, \tau}$ for all $0 \le \rho \le \sig \le \tau \le 1$. Moreover, $\sfP^*_{\sig, \tau} \nu_\sig = \nu_\tau$ for all $0 \le \sig \le \tau \le 1$, in the sense that
\begin{equation}\label{eq:p_adjoint}\E_{\nu_\tau}\Brack{f} = \E_{\nu_\sig}\Brack{\sfP_{\sig, \tau} f} \text{ for all } f \in \calF(\R^d).\end{equation}
\end{lemma}
\begin{proof}
By applying \eqref{eq:sg_def} twice,
\begin{align*}
    \sfP_{\rho, \sig} \sfP_{\sig, \tau}f(\mathbf{x})  &= \exp\Par{V_\rho(\vx)}\mathbb{E}_{\mathbf{y}\sim \mathcal{N}(\0_d, (\sigma-\rho)\id_d)}\left[\exp\Par{-V_\sigma(\vx+\mathbf{y})} \exp\Par{V_{\sigma}(\vx+\mathbf{y})}\right. \\
    &\left. \quad\cdot \mathbb{E}_{\vz\sim \mathcal{N}(\0_d, (\tau-\sigma)\id_d)}\Brack{\exp\Par{-V_\tau(\vx+\mathbf{y}+\vz)}f(\vx+\mathbf{y}+\vz)}\right]\\
    &= \exp\Par{V_{\rho}(\mathbf{x})}\mathbb{E}_{\mathbf{y}\sim \mathcal{N}(\0_d, (\sigma-\rho)\id_d)}\Brack{\mathbb{E}_{\vz\sim \mathcal{N}(\0_d, (\tau-\sigma)\id_d)}\Brack{\exp\Par{-V_\tau(\vx+\mathbf{y}+\vz)}f(\vx+\mathbf{y}+\vz)}}\\
    &= \exp\Par{V_{\rho}(\vx)}\mathbb{E}_{\mathbf{z}\sim \mathcal{N}(\0_d, (\tau-\rho)\id_d)}\Brack{\exp\Par{-V_\tau(\vx+\vz)}f(\vx+\vz)} = \sfP_{\rho, \tau}f(\vx),
\end{align*}
where we use the fact that if $\vx_1\sim \mathcal{N}(\0_d, \msig_1), \vx_2\sim \mathcal{N}(\0_d, \msig_2)$ then $\vx_1+\vx_2\sim \mathcal{N}(\0_d, \msig_1+\msig_2)$.

Next, note that for all $\tau \in [0, 1]$,
\begin{align*}
    (2\pi\tau)^{-\frac d 2}\int \exp\Par{-V_\tau(\vx) - \frac 1 {2\tau}\norm{\vx}_2^2} \, \dd \vx &= \mathbb{E}_{\mathbf{x}\sim\mathcal{N}(\0_d, \tau \id_d)}\Brack{\exp\Par{-V_\tau(\mathbf{x})}}\\
    &= \mathbb{E}_{\mathbf{x}\sim\mathcal{N}(\0_d, \tau \id_d)}\Brack{\mathbb{E}_{\vz\sim\mathcal{N}(\0_d, (1-\tau)\id_d)}\Brack{\exp\Par{-V_1(\vx+\vz)}} }\\
    &=\mathbb{E}_{\vz\sim \mathcal{N}(\0_d, \id_d)}\Brack{\exp\Par{-V_1(\vz)}} =\exp\Par{-V_0(\0_d)},
\end{align*}
so the normalizing constant of $\nu_\tau$ is $(2\pi\tau)^{\frac d 2}\exp\Par{V_0(\0_d)}$. Thus, for all $f \in \calF(\R^d)$,
\begin{equation}\label{eq:mean_f}
\begin{aligned}
\sfP_{0,\tau} f(\0_d) 
&= \exp\Par{V_0(\0_d)} \cdot \mathbb{E}_{\vz\sim \mathcal{N}(\0_d, \tau \id_d)}\Brack{\exp(-V_\tau(\vz))f(\vz)}\\
&=\frac{(2\pi\tau)^{\frac d 2}}{\int \exp\Par{-V_\tau(\vz) -\frac{1}{2\tau}\norm{\vz}^2 }\,\D\vz} \cdot \frac{\int \exp\Par{-V_\tau(\vz) -\frac{1}{2\tau}\norm{\vz}^2 }f(\vz)\,\D\vz }{(2\pi\tau)^{\frac d 2}} = \mathbb{E}_{\nu_\tau}[f].
\end{aligned}
\end{equation}
The second claim now follows as
\[\mathbb{E}_{\nu_\tau}[f]=\sfP_{0, \tau}f(\0_d)=\sfP_{0, \sigma}\sfP_{\sigma, \tau}f(\0_d)=\mathbb{E}_{\nu_\sigma}\Brack{\sfP_{\sigma, \tau}f}.\qedhere\]
\end{proof}

We next describe $\{\sfP_{\sig, \tau}\}_{0 \le \sig \le \tau \le 1}$ via its infinitesimal generators $\{\sfL_\tau\}_{\tau \in [0, 1]}$.

\begin{lemma}\label{lem:polchinski_gen}
Define a family of operators $\{\sfL_\tau\}_{\tau \in [0, 1]}$ that act on $f \in \calF(\R^d)$ via
\begin{equation}\label{eq:polchinski_gen_def}\sfL_\tau f = \half \Delta f - \inprod{\nabla V_\tau}{\nabla f}, \text{ for all } \tau \in [0, 1].\end{equation}
Then, we have for all $0 \le \sig \le \tau \le 1$ and $f \in \calF(\R^d)$ that
\begin{equation}\label{eq:ltau_polchinski}\partial_\sig \sfP_{\sig, \tau} f = -\sfL_\sig \sfP_{\sig, \tau} f,\quad \partial_\tau \sfP_{\sig, \tau} f = \sfP_{\sig,\tau} \sfL_\tau  f,\quad 
\partial_\tau\mathbb{E}_{\nu_\tau}[f]=\mathbb{E}_{\nu_\tau}[\sfL_\tau f].\end{equation}
\end{lemma}
\begin{proof}
Let $\gamma_\tau(\vz) \propto \exp(-\frac 1 {2\tau}\norm{\vz}_2^2)$ be the density of $\Nor(\0_d, \tau \id_d)$, and recall the heat equation, $\partial_\tau \gamma_\tau = \half \Delta \gamma_\tau$. Then, letting $(f \ast g)(\vx) \defeq \int f(\vx - \vz) g(\vz)\, \dd \vz$ denote convolution, we have
\begin{equation}\label{eq:heat_conv}
\partial_\tau \Par{f \ast \gamma_{1 - \tau}}(\vx) = \partial_\tau\int f(\vz) \gamma_{1 - \tau}(\vx - \vz)\dd \vz = -\half \Delta (f \ast \gamma_{1 - \tau})(\vx).
\end{equation}
Next, define
$Z_\tau(\vx) \defeq \exp(-V_\tau(\vx)) = \mathbb{E}_{\vz \sim \mathcal{N}(\mathbf{0}_d, (1-\tau)\id_d)}[ \exp(-V_1(\vx - \vz)) ]$, where we used symmetry of $\vz$ in the last equality. Because $Z_\tau = \exp(-V_1) \ast \gamma_{1 - \tau}$, applying \eqref{eq:heat_conv} yields
    \begin{align*}\partial_\tau Z_\tau(\vx) = -\frac{1}{2}\Delta Z_\tau(\vx).\end{align*}
Now applying the chain rule to the definition of $Z_\tau$ yields
\begin{align*}
\nabla Z_\tau = -\nabla V_\tau \exp(-V_\tau) \implies \Delta Z_\tau = \Par{-\Delta V_\tau + \norm{\nabla V_\tau}_2^2}\exp(-V_\tau).
\end{align*}
Combining the above two displays, we derive
\begin{equation}\label{eq:partial_V}
\partial_\tau V_\tau = -\frac{\partial_\tau Z_\tau}{Z_\tau} = \frac{\Delta Z_\tau}{2Z_\tau} = -\half \Delta V_\tau + \half \norm{\nabla V_\tau}_2^2.
\end{equation}
Finally, 
    \begin{align*}
        \partial_\sig \sfP_{\sigma, \tau}f
        &=\left(\partial_\sig V_\sigma\right)\sfP_{\sigma, \tau}f- \half \exp(V_\sig)\Delta\Par{\exp(-V_\sig)\sfP_{\sig, \tau} f}\\
        &=\left(\partial_\sig V_\sigma\right)\sfP_{\sigma, \tau}f - \half \exp(V_\sig)\nabla \cdot \Par{ - \nabla V_\sig\exp(-V_\sig)\sfP_{\sig,\tau} f + \exp(-V_\sig)\nabla \sfP_{\sig,\tau} f} \\
        &=\left( \partial_\sig V_\sigma\right)\sfP_{\sigma, \tau}f + \frac{1}{2}(\Delta V_\sigma)\sfP_{\sigma, \tau}f -\frac{1}{2}\norm{\nabla V_\sigma}_2^2 \sfP_{\sigma, \tau}f-\frac{1}{2}\Delta \sfP_{\sigma, \tau}F+\inprods{\nabla V_\sigma}{\nabla \sfP_{\sigma, \tau}f} \\
        &=-\frac{1}{2}\Delta \sfP_{\sigma, \tau}f+\inprods{\nabla V_\sigma}{\nabla \sfP_{\sigma, \tau}f} =-\sfL_\sigma \sfP_{\sigma, \tau}f,
    \end{align*}
    where we used the chain rule and \eqref{eq:heat_conv} in the first line, and \eqref{eq:partial_V} in the last.
    The claim about $\partial_\tau \sfP_{\sigma, \tau}f$ follows similarly.  For the last claim, applying \eqref{eq:p_adjoint} twice for $\sigma\leq \tau$,
    \[\partial_\tau \mathbb{E}_{\nu_\tau}[f]= \partial_\tau \mathbb{E}_{\nu_\sigma}[\sfP_{\sigma, \tau}f]=\mathbb{E}_{\nu_\sigma}\left[ \partial_\tau \sfP_{\sigma, \tau}f\right]=\mathbb{E}_{\nu_\sigma}[\sfP_{\sigma, \tau}\sfL_\tau f]=\mathbb{E}_{\nu_\tau}[\sfL_\tau f].  \qedhere\]
\end{proof}

The evolution of the renormalized potential $V_\tau$ in \eqref{eq:partial_V} is known as the \emph{Polchinski equation}.  Further, note that \eqref{eq:ltau_polchinski} is consistent with our earlier definition of infinitesimal generators in \eqref{eq:generator_tau}, because the following \emph{Kolmogorov equations} hold: using \eqref{eq:generator_tau} as our definition,
\begin{align*}
\partial_\sig \sfP_{\sig, \tau} f &= \lim_{\eta \searrow 0} \frac{\sfP_{\sig + \eta, \tau} f - \sfP_{\sig, \tau} f}{\eta} = \lim_{\eta \searrow 0} \frac{\sfP_{\sig, \sig + \eta} (-\sfP_{\sig + \eta, \tau} f) - (-\sfP_{\sig + \eta, \tau} f)}{\eta} = -\sfL_\sig \sfP_{\sig, \tau} f, \\
\partial_\tau \sfP_{\sig, \tau} f &= \lim_{\eta \searrow 0} \frac{\sfP_{\sig, \tau + \eta} f - \sfP_{\sig, \tau} f}{\eta} = \sfP_{\sig, \tau} \lim_{\eta \searrow 0} \frac{\sfP_{\tau, \tau + \eta} f - f}{\eta} = \sfP_{\sig, \tau} \sfL_\tau f.
\end{align*}

To complete our exposition of the Polchinski semigroup, we note that the infinitesimal generators $\{\sfL_\tau\}_{\tau \in [0, 1]}$ induce a SDE on particles $\{\vv_\tau \in \R^d\}_{\tau \in [0, 1]}$, such that $\vv_\tau \sim \nu_\tau$.

\begin{lemma}\label{lem:polchinski_sde_derive}
For all $\vv \in \R^d$ and $\tau \in [0, 1]$, define an induced \emph{fluctuation measure}
\begin{equation}\label{eq:fluctuation2}
    \pi_\tau^{\vv}(\mbf{x})\propto \exp\left(\frac{1}{1-\tau}\inprods{\vv}{\mbf{x}}-\frac{\tau}{2(1-\tau)}\norm{\mbf{x}}_2^2\right)\pi_0(\mbf{x}).
\end{equation}
Consider the following SDE, where $\{\vW_\tau\}_{\tau \in [0, 1]}$ is a Wiener process:
\begin{equation}\label{eq:polchinski_sde}
\dd \vv_\tau = -\frac 1 {1 - \tau}(\vv_\tau - \vm_\tau) \, \dd \tau + \dd \vW_\tau, \text{ where } \vm_\tau \defeq \E_{\vx \sim \pi_\tau^{\vv_\tau}}[\vx],\; \mathbf{v}_0=\mathbf{0}_d.
\end{equation}
The infinitesimal generator of \eqref{eq:polchinski_sde} is $\sfL_\tau$ in \eqref{eq:polchinski_gen_def}, and for any $\tau \in [0, 1]$, we have $\vv_\tau \sim \nu_\tau$.
\end{lemma}
\begin{proof}
We first claim that 
\begin{equation}\label{eq:renorm_expect}
\E_{\pi_\tau^{\vv}}[f] = \sfP_{\tau, 1} f(\vv) \text{ for all } f \in \calF(\R^d),
\end{equation}
which follows from
\begin{equation}\label{eq:pt1f}
\begin{aligned}
\sfP_{\tau, 1} f(\vv) &= \frac{\E_{\vz \sim \Nor(\0_d, (1 - \tau)\id_d)}\Brack{\exp\Par{-V_1(\vv + \vz)} f(\vv + \vz)}}{\E_{\vz \sim \Nor(\0_d, (1 - \tau) \id_d)}\Brack{\exp\Par{-V_1(\vv + \vz)}}} \\
&= \frac{\int \exp\Par{-V_1(\vx) - \frac 1 {2(1-\tau)}\norm{\vx - \vv}_2^2} f(\vx) \,\dd \vx}{\int \exp\Par{-V_1(\vx) - \frac 1 {2(1-\tau)}\norm{\vx - \vv}_2^2}\, \dd \vx},
\end{aligned}
\end{equation}
where the second line substituted $\vx = \vv + \vz$. Next, Lemma~\ref{lem:fokker_planck} shows that if
\begin{equation}\label{eq:grad_sde}\D \vv_\tau=-\nabla V_\tau(\vv_\tau)\,\D\tau+\D \vW_\tau,\end{equation}
the generator of \eqref{eq:grad_sde} is precisely $\sfL_\tau$, and from $\sfP^*_{\sigma, \tau}\nu_\sigma=\nu_\tau$ (Lemma~\ref{lem:semigroup}) we obtain the final statement, assuming we can show \eqref{eq:grad_sde} is equivalent to \eqref{eq:polchinski_sde}. To obtain this equivalence, note that
\[\nabla V_\tau(\vv)=\frac{\mathbb{E}_{\vz\sim \mathcal{N}(\0_d, (1-\tau)\id_d)}\Brack{\exp\Par{-V_1(\vv+\vz)}\nabla V_1(\vv+\vz)}}{\mathbb{E}_{\vz\sim \mathcal{N}(\0_d, (1-\tau)\id_d)}\Brack{\exp(-V_1(\vv+\vz))}}=(\sfP_{\tau, 1}\nabla V_1)(\vv)=\mathbb{E}_{\pi_\tau^{\vv}}\Brack{\nabla V_1},\]
where the last equality uses our earlier derivation \eqref{eq:pt1f} coordinatewise. We conclude by recalling that for any measure $\mu$ with sufficiently fast decay, $
\E_{\vx \sim \mu}[\nabla \log \mu(\vx)] = \int \nabla \mu(\vx) \,\dd \vx = \0_d$,
where we integrated by parts coordinatewise. Applying this fact with $\mu = \pi_\tau^{\vv_\tau}$ completes the proof:
\begin{align*}
\E_{\vx \sim \pi_\tau^{\vv_\tau}}\Brack{-\nabla V_1(\vx) + \frac 1 {1-\tau}(\vv_\tau - \vx)} = \0_d \implies \E_{\pi_\tau^{\vv_\tau}}\Brack{\nabla V_1} = \frac 1 {1 - \tau} \Par{\vv_\tau - \vm_\tau}.
\end{align*}
Alternatively, this can be seen by an application of Tweedie's formula (Lemma \ref{lem:tweedie}), noting that $-\nabla V_\tau(\vx)$ is nothing more than $\nabla \log \mu(\vx)$, where
\[\mu(\vx)\propto \int \exp\left(-\frac{1}{2(1-\tau)}\norm{\vx-\vz}^2\right)\exp\Par{-V_1(\vz)}\D\vz, \]
employing a change of variables.
\end{proof}

We are finally ready to connect the Polchinski semigroup to stochastic localization.

\begin{theorem}\label{thm:tilt_renormal}
The processes $\{\vc_t\}_{t \ge 0}$ and $\{\vv_\tau\}_{\tau \in [0, 1]}$ in \eqref{eq:tilt_def}, \eqref{eq:polchinski_sde} satisfy $\vc_{t} = \frac 1 {1-\tau}\vv_\tau$, where $t = \frac \tau {1 - \tau}$, and the induced $\pi_t$ in \eqref{eq:tilt_def} and $\pi_\tau^{\vv_\tau}$ in \eqref{eq:fluctuation2} are identical under this reparameterization.
\end{theorem}
\begin{proof}
We begin by equating the measures in \eqref{eq:tilt_def} and \eqref{eq:fluctuation2}: starting from \eqref{eq:tilt_def},
\begin{align*}
\pi_t(\vx) \propto \exp\Par{\inprod{\vc_t}{\vx} - \frac t 2 \norm{\vx}_2^2}\pi_0(\vx) 
= \exp\Par{\frac 1 {1-\tau}\inprod{\vv_\tau}{\vx} - \frac \tau {2(1-\tau)} \norm{\vx}_2^2}\pi_0(\vx),
\end{align*}
which matches the definition of $\pi_\tau^{\vv_\tau}$ in \eqref{eq:fluctuation2}. Therefore, $\vm_t$ in \eqref{eq:tilt_def} and $\vm_\tau$ in \eqref{eq:polchinski_sde} have the same definition, under the time change $t = \frac \tau {1 - \tau}$. Next, letting $\vu_\tau \defeq \frac 1 {1 - \tau}\vv_\tau$ where $\vv_\tau$ follows \eqref{eq:polchinski_sde},
\begin{equation}\label{eq:scaled_sde}
\begin{aligned}
\partial_\tau \vu_\tau &= \frac 1 {(1 - \tau)^2} \vv_\tau + \frac 1 {1-\tau}\frac{\partial}{\partial \tau} \vv_\tau = \frac 1 {(1 - \tau)^2} \vm_\tau + \frac 1 {1 - \tau} \dd \vW_\tau.
\end{aligned}
\end{equation}
At this point, applying the time change formula\footnote{This formula can be arrived at via a similar calculation as in Lemma~\ref{lem:timechange}, that changes but does not reverse time.} for It\^o diffusions (cf.\ Theorem 8.5.1, \cite{Oksendal00}), with $\partial_\tau  \frac 1 {1 - \tau} = \frac 1 {(1 - \tau)^2}$ shows equivalence of $\vu_\tau$ in \eqref{eq:scaled_sde} and $\vc_t$ in \eqref{eq:tilt_def}.
\end{proof}

%% file: sections/schrodinger.tex
\section{Static Schr\"odinger bridge}\label{sec:schrodinger}

In this section, we place stochastic localization in the context of the Schr\"odinger bridge \cite{Schro1931, Schro1932}. We again restrict time to $\tau \in [0, 1]$, and follow the presentation of \cite{Leonard14} throughout.

We begin by describing a general formulation of the Schr\"odinger bridge optimization problem. 

\begin{tcolorbox}[colback=blue!10, colframe=blue!50!black, boxrule=0.5pt, arc=2mm]
\begin{perspective}\label{perspective:bridge}
Let $\mu, \pi \in \calP(\R^d)$, and let $R \in \calP([0, 1] \times \R^d)$ be a reference \emph{path measure} on $\calC([0, 1] \times \R^d)$. The induced \emph{static Schr\"odinger bridge} problem is:
\begin{equation}\label{eq:schrodinger}\tag{$\msf{SSB}$}
    \inf_{P \in \mathcal P([0, 1] \times \R^d)} \KL{P}{R} 
    \textrm{ such that } P_0 = \mu,\; P_1 = \pi.
\end{equation}
\end{perspective}
\end{tcolorbox}

In other words, we seek the path measure on $\calC([0, 1] \times \R^d)$, with prescribed starting measure $P_0 = \mu$ and ending measure $P_1 = \pi$, that is as close as possible in KL divergence to a given path measure $R$. It may be helpful to keep the running example of $R$ being the Wiener path measure in mind, as we specialize calculations at the end of the section to this case.

We now introduce some helpful notation. When $P \in \calP([0, 1] \times \R^d)$ is a path measure, $P_{01}$ denotes the marginal density of its endpoints $(P_0, P_1)$, and similarly we denote the endpoints of $\vp_{[0, 1]} \in \calC([0, 1] \times \R^d)$ by $\vp_{01} \in \R^d \times \R^d$. For $(\vx, \vy) \in \R^d \times \R^d$, we let $P_{(0, 1) \mid (\vx, \vy)}$ denote the conditional density of $\vp_{(0, 1)} \mid \vp_{01} = (\vx, \vy)$ for $\vp_{[0, 1]} \sim P$, i.e., the intermediate path given the endpoints. 
Finally, we let $\Gamma(\mu, \pi)$ denote the set of \emph{couplings} of $\mu$ and $\pi$, i.e., measures $\gamma$ on $\R^d \times \R^d$ whose left and right marginals, denoted $\gamma_0$ and $\gamma_1$, respectively equal $\mu$ and $\pi$.

The first observation regarding \eqref{eq:schrodinger} is that to find its solution (when it is attainable), it suffices to find the coupling of the marginals $P_0 = \mu$ and $P_1 = \pi$ closest to $R_{01}$ in KL divergence. This coupling is then extended to paths on $[0, 1]$ via matching conditional distributions.

\begin{lemma}\label{lem:endpoints_suffice}
Let $P^\star$ optimally solve \eqref{eq:schrodinger}, and let $\gamma^\star \in \Gamma(\mu, \pi)$ optimally solve
\begin{equation}\label{eq:schro-static}
    \inf_{\gamma \in \Gamma(\mu, \pi)} \KL{\gamma}{R_{01}}.
\end{equation}
Then $P^\star$ and $\gamma^\star$ are related as follows: for all $\vp_{[0, 1]} \in \calC([0, 1] \times \R^d)$ and $(\vx, \vy) \in \R^d \times \R^d$,
\begin{align*}
P^\star_{01} = \gamma^\star,\quad P_{(0, 1) \mid (\vx, \vy)}\Par{\vp_{(0, 1)} \mid \vp_{01} = (\vx, \vy)} = R_{(0, 1) \mid (\vx, \vy)}\Par{\vp_{(0, 1)} \mid \vp_{01} = (\vx, \vy)}.
\end{align*}
\end{lemma}
\begin{proof}
    The chain rule for the KL divergence gives
    \[\KL{P}{R}=\KL{P_{01}}{R_{01}}+\mathbb{E}_{(\vx, \vy) \sim P_{01}}\Brack{\KL{P_{(0, 1) \mid (\vx, \vy)}}{R_{(0, 1) \mid (\vx, \vy)}}}.\]
    We can choose $P_{01}$ and $P_{(0, 1) \mid (\vx, \vy)}$ separately: the first term above is the objective in \eqref{eq:schro-static}, and the second is optimized (in fact zero) by setting $P_{(0, 1) \mid (\vx, \vy)} = R_{(0, 1) \mid (\vx, \vy)}$ for all $(\vx, \vy) \in \R^d \times \R^d$.
\end{proof}

Next, we characterize the solution to the (convex) problem \eqref{eq:schro-static} via taking its dual.

\begin{lemma}\label{lem:endpoint_sol}
Let $\gamma^\star \in \Gamma(\mu, \pi)$ optimally solve \eqref{eq:schro-static}.
There exist $f, g: \R^d \to \R_{\ge 0}$ such that 
\begin{equation}\label{eq:gam-rep}
\gamma^\star(\vx, \vy) = R_{01}(\vx, \vy)f(\vx)g(\vy)
\end{equation}
for all $(\vx, \vy) \in \R^d \times \R^d$, and the following \emph{Schr\"odinger system} is satisfied:
\begin{equation}\label{eq:schro-system}
\frac{\mu(\vx)}{R_0(\vx)} = f(\vx)\E_{(\vx, \vy) \sim R_{01}}\Brack{g(\vy) \mid \vx},\quad \frac{\pi(\vy)}{R_1(\vy)} = g(\vy)\E_{(\vx, \vy) \sim R_{01}}\Brack{f(\vx) \mid \vy}.
\end{equation}
\end{lemma}
\begin{proof}
We begin by writing the Lagrangian of \eqref{eq:schro-static}: denoting $\rho \defeq R_{01}$ for short, it is
\begin{gather*}
\KL{\gamma}{\rho}
+ \int \phi(\vx)\Par{\mu(\vx) - \int \gamma(\vx,\vy)\,\dd\vy} \, \dd \vx + \int \psi(\vy)\Par{\pi(\vy) - \int \gamma(\vx,\vy) \,\dd\vx} \,\dd\vy,
\end{gather*}
where $\phi, \psi: \R^d \to \R$ are absolutely integrable with respect to $\mu, \pi$ respectively. To minimize this over $\gamma$,\footnote{Strong duality holds for this problem under mild regularity conditions, justifying the exchanging of $\min$ and $\max$; we refer the reader to Appendix A of \cite{Leonard14} for technicalities regarding this point.} we derive using the Donsker--Varadhan variational formula:
\begin{align*}
&\min_{\gamma \in \calP(\R^d \times \R^d)}\Brace{\KL{\gamma}{\rho} - \iint (\phi(\vx) + \psi(\vy))\gamma(\vx,\vy)\,\dd\vx\,\dd\vy} \\
&\qquad\qquad= -\log\Par{\iint\exp\Par{\phi(\vx) + \psi(\vy)} \rho(\vx, \vy)\,\dd\vx\,\dd\vy}
\end{align*}
where $\calP(\R^d \times \R^d)$ is the set of probability measures over $\R^d \times \R^d$. Thus \eqref{eq:schro-static} is equivalent to
\begin{align*}
\max_{\phi, \psi} \int \phi(\vx) \mu(\vx) \,\dd\vx + \int \psi(\vy) \pi(\vy)\,\dd\vy - \log\Par{\iint\exp\Par{\phi(\vx) + \psi(\vy)} \rho(\vx, \vy)\,\dd\vx\,\dd\vy}.
\end{align*}
Now letting $\phi^\star$, $\psi^\star$ optimize the above expression, and writing $f = \exp(\phi^\star)$, $g = \exp(\psi^\star)$, first-order optimality conditions now show the optimizer of \eqref{eq:schro-static} satisfies
$\gamma^\star(\vx,\vy) = \rho\Par{\vx,\vy}f(\vx)g(\vy)$, 
as claimed. Finally, to compare $\vx$-marginals between $\gamma^\star$ and $\rho$, we compute
\begin{align*}
\frac{\mu(\vx)}{R_0(\vx)} = \frac{\int \rho(\vx,\vy) f(\vx)g(\vy)\,\dd \vy}{\int \rho(\vx,\vy)\,\dd\vy} = f(\vx) \E_{(\vx, \vy) \sim \rho}\Brack{g(\vy) \mid \vx},
\end{align*}
which matches the expression in \eqref{eq:schro-system}; the calculation for $\vy$-marginals follows similarly.
\end{proof}

The factorizations in \eqref{eq:gam-rep}, \eqref{eq:schro-system} extend more generally to characterize the marginal measure $P_\tau$ for all $\tau \in [0, 1]$, when the reference path measure $R$ is \emph{Markov}. Recall that $R$ is Markov if for all times $\tau \in [0, 1]$, we have the conditional independence property
\begin{equation}\label{eq:markov_path}\vp_{[0,\tau]} \perp \vp_{[\tau, 1]} \mid \vp_\tau, \text{ for } \vp_{[0, 1]} \sim R.\end{equation}

\begin{lemma}\label{lem:pstar_markov}
Suppose that $R$ is a Markov reference path measure, i.e., \eqref{eq:markov_path} holds, and define $f, g: \R^d \to \R_{\ge 0}$ as in \eqref{eq:gam-rep}. For all $\tau \in [0, 1]$, define
\begin{equation}\label{eq:factor_tau}f_\tau(\vz) \defeq \E_{\vp_{[0, 1]} \sim R}\Brack{f(\vp_0) \mid \vp_\tau = \vz},\quad g_\tau(\vz) \defeq \E_{\vp_{[0, 1]} \sim R}\Brack{g(\vp_1) \mid \vp_\tau = \vz},\end{equation}
let $P^\star$ optimally solve \eqref{eq:schrodinger} and define $P^\star_{0 \mid \tau}$, $P^\star_{1 \mid \tau}$ to be the laws of $\vp_0 \mid \vp_\tau$, $\vp_1 \mid \vp_\tau$ for $\vp_{[0, 1]} \sim P^\star$ (and $R_{0 \mid \tau}$, $R_{1 \mid \tau}$ similarly for $\vp_{[0, 1]} \sim R$). Then $P^\star$ is Markov, and for all $0 \le \sig \le \tau \le 1$,
\begin{equation}\label{eq:pstar_forms}
\begin{gathered}
P^\star(\vp_{[0, 1]}) = f(\vp_0) g(\vp_1)R(\vp_{[0, 1]}) ,\quad P^\star_\tau(\vz) = f_\tau(\vz)g_\tau(\vz)R_\tau(\vz),\\
P^\star_{\sig \mid \tau}(\vs \mid \vt) = \frac{f_\sig(\vs)}{f_\tau(\vt)} R_{\sig \mid \tau}(\vs \mid \vt),\quad P^\star_{\tau \mid \sig}(\vt \mid \vs) = \frac{g_\tau(\vt)}{g_\sig(\vs)}R_{\tau \mid \sig}(\vt \mid \vs).
\end{gathered}
\end{equation}
\end{lemma}
\begin{proof}
Combining Lemmas~\ref{lem:endpoints_suffice} and~\ref{lem:endpoint_sol}, and following notation \eqref{eq:factor_tau}, we have shown that $P^\star_0 = f_0 g_0 R_0$, $P^\star_1 = f_1 g_1 R_1$, and $P^\star_{0,1}(\vx, \vy) = R_{0,1}(\vx, \vy) f(\vx)g(\vy)$. Multiplying both sides of the last equation by $P^\star_{(0, 1) \mid (\vx, \vy)} = R_{(0, 1) \mid (\vx, \vy)}$ from Lemma~\ref{lem:endpoints_suffice} then gives the first claim in \eqref{eq:pstar_forms}.

We next prove that $P^*$ is Markov.  Let $0\leq \tau \leq 1$ and consider bounded functions $a, b$ such that $a$ is measurable with respect to (the sigma algebra generated by) $\vp_{[0, \tau]}$, and  $b$ is measurable with respect to $\vp_{[\tau, 1]}$.  Then
\begin{align*}
    \mathbb{E}_{\vp_{[0, 1]}\sim P^*}\Brack{a(\vp_{[0, \tau]}) b(\vp_{[\tau, 1]})\mid \vp_\tau}
    &=\frac{\mathbb{E}_{\vp_{[0, 1]}\sim R}\Brack{f(\vp_0) a(\vp_{[0, \tau]})b(\vp_{[\tau, 1]}) g(\vp_1)\mid \vp_\tau}}{\mathbb{E}_{\vp_{[0, 1]}\sim R}\Brack{f(\vp_0)g(\vp_1)\mid \vp_\tau}} \\
    &=\frac{\mathbb{E}_{\vp_{[0, 1]}\sim R}\Brack{f(\vp_0) a(\vp_{[0, \tau]}) \mid \vp_\tau} \cdot \mathbb{E}_{\vp_{[0, 1]}\sim R}\Brack{b(\vp_{[\tau, 1]}) g(\vp_1)\mid \vp_\tau}}{\mathbb{E}_{\vp_{[0, 1]}\sim R}\Brack{f(\vp_0)\mid \vp_\tau} \cdot \mathbb{E}_{\vp_{[0, 1]}\sim R}\Brack{g(\vp_1)\mid \vp_\tau}} 
    \\
    &=\mathbb{E}_{\vp_{[0, 1]}\sim P^*}[a(\vp_{[0, \tau]})\mid \vp_\tau]\cdot \mathbb{E}_{\vp_{[0, 1]}\sim P^*}[b(\vp_{[ \tau, 1]})\mid \vp_\tau],
\end{align*}
where the first and last equalities are general results about conditioning, and the second equality uses the Markov property of $R$. We note that we do not divide by $0$, $P^\star$-almost surely.

The second claim in \eqref{eq:pstar_forms} applies the Markov property of $R$:
\begin{align*}\frac{P^\star_\tau(\vz)}{R_\tau(\vz)} &= \E_{\vp_{[0, 1]} \sim R}\Brack{\frac{P^\star(\vp_{[0, 1]})}{R(\vp_{[0, 1]})} \mid \vp_\tau = \vz} \\
&= \E_{\vp_{[0, 1]} \sim R}\Brack{f(\vp_0)g(\vp_1) \mid \vp_\tau = \vz} \\
&= \E_{\vp_{[0, 1]} \sim R}\Brack{f(\vp_0)\mid\vp_\tau = \vz}\cdot \E_{\vp_{[0, 1]} \sim R}\Brack{g(\vp_1) \mid \vp_\tau = \vz} = f_\tau(\vz) g_\tau(\vz).\end{align*}
Similarly, for any $0 \le \sig \le \tau \le 1$,
\begin{align*}
    \frac{P^\star_{\sigma, \tau}(\vs, \vt)}{R_{\sigma, \tau}(\vs, \vt)}
    &=\mathbb{E}_{\vp_{[0, 1]}\sim R}\Brack{\frac{P^*(\vp_{[0, 1]})}{R(\vp_{[0, 1]})}\mid \vp_{\sigma, \tau}=(\vs, \vt)}\\
    &=\mathbb{E}_{\vp_{[0, 1]}\sim R}\Brack{f(\vp_0)g(\vp_1)\mid \vp_{\sigma, \tau}=(\vs, \vt)}\\
    &=\mathbb{E}_{\vp_{[0, 1]}\sim R}\Brack{f(\vp_0)\mid \vp_\sigma =\vs}\cdot \mathbb{E}_{\vp_{[0, 1]}\sim R}\Brack{g(\vp_1)\mid \vp_\tau =\vt} =f_\sigma(\vs)g_\tau(\vt),
\end{align*}
where the third equality applies that $R$ is Markov.  The last two parts of \eqref{eq:pstar_forms} now follow:
\begin{align*}
    & P^\star_{\sigma\mid \tau}(\vs\mid \vt)=\frac{P^\star_{\sigma, \tau}(\vs, \vt)}{P^\star_{\tau}(\vt)}=\frac{f_\sigma(\vs)}{f_\tau(\vt)}R_{\sigma\mid\tau}(\vs \mid \vt),\\
    & P^\star_{\tau\mid \sigma}(\vt\mid \vs)=\frac{P^\star_{\sigma, \tau}(\vs, \vt)}{P^\star_{\sigma}(\vs)}=\frac{g_\tau(\vt)}{g_\sigma(\vs)}R_{\tau\mid\sigma}(\vt \mid \vs).\qedhere
\end{align*} 
\end{proof}

Our next step is to derive the infinitesimal generator of $P^\star$. Consistent with \eqref{eq:generator_tau}, we define the infinitesimal generator corresponding to a path measure $P \in \calP([0, 1] \times \R^d)$ at time $\tau \in [0, 1]$ by:
\begin{equation}\label{eq:gen_P}
\sfL_\tau^P u(\vz) \defeq \lim_{\eta \searrow 0} \frac{\E_{\vp_{[0, 1]} \sim P}\Brack{u(\vp_{\tau + \eta}) - u(\vp_\tau) \mid \vp_\tau = \vz}}{\eta}, 
\end{equation}
and we extend \eqref{eq:cdc} slightly to the setting where $\{v_\tau\}_{\tau \in [0, 1]}$ are time-inhomogeneous functions:
\begin{equation}\label{eq:cdc_P}
\sfGam_\tau^P(u, v_\tau)(\vz) \defeq \half \lim_{\eta \searrow 0} \frac{\E\Brack{(u(\vp_{\tau + \eta}) - u(\vp_\tau))(v_{\tau + \eta}(\vp_{\tau + \eta}) - v_\tau(\vp_\tau)) \mid \vp_\tau = \vz}}{\eta}.
\end{equation}

\begin{lemma}
Following notation in Lemma~\ref{lem:pstar_markov}, \eqref{eq:gen_P}, and \eqref{eq:cdc_P}, for all $u \in \calF(\R^d)$ and $\vz \in \R^d$,
\begin{equation}\label{eq:P-generator}\sfL_\tau^{P^\star} u(\vz) = \sfL_\tau^R u(\vz) + \frac{2\sfGam_\tau^R(g_\tau, u)(\vz)}{g_\tau(\vz)}. \end{equation}
\end{lemma}
\begin{proof}
First, by using \eqref{eq:pstar_forms} and the Markov property of $R$,
\begin{align*}
\sfL_\tau^{P^\star} u(\vz) &=  \lim_{\eta \searrow 0} \frac 1 \eta \cdot\frac{\E_{\vp_{[0, 1]} \sim R}\Brack{f(\vp_0)(u(\vp_{\tau + \eta}) - u(\vp_\tau)) g(\vp_1)\mid \vp_\tau = \vz}}{ \E_{\vp_{[0, 1]} \sim R}\Brack{f(\vp_0) g(\vp_1) \mid \vp_\tau = \vz}} \\
&= \lim_{\eta \searrow 0} \frac{f_\tau(\vz)\E_{\vp_{[0, 1]} \sim R}\Brack{(u(\vp_{\tau + \eta}) - u(\vp_\tau)) g(\vp_1)\mid \vp_\tau = \vz}}{\eta f_\tau(\vz) g_\tau(\vz)} \\
&= \frac 1 {g_\tau(\vz)} \lim_{\eta \searrow 0} \frac{\E_{\vp_{[0, 1]} \sim R}\Brack{(u(\vp_{\tau + \eta}) - u(\vp_\tau)) g(\vp_1)\mid \vp_\tau = \vz}}{\eta}.
\end{align*}
However, we also may expand
\begin{align*}
\E_{\vp_{[0, 1]} \sim R}&\Brack{(u(\vp_{\tau + \eta}) - u(\vp_\tau)) g(\vp_1)\mid \vp_\tau = \vz} \\
&= \int (u(\vw) - u(\vz)) g_{\tau + \eta}(\vw) R_{\tau + \eta \mid \tau}(\vw \mid \vz)\,\dd\vw \\
&= \int (u(\vw) - u(\vz)) g_{\tau}(\vz) R_{\tau + \eta \mid \tau}(\vw \mid \vz)\,\dd\vw \\
&\qquad + \int (u(\vw) - u(\vz)) (g_{\tau + \eta}(\vw) - g_\tau(\vz)) R_{\tau + \eta \mid \tau}(\vw \mid \vz)\,\dd\vw \\
&= g_\tau(\vz) \E_{\vp_{[0, 1]} \sim R}\Brack{u\Par{\vp_{\tau + \eta}} - u\Par{\vp_{\tau} }\mid \vp_\tau = \vz} \\
&\qquad + \E_{\vp_{[0, 1]} \sim R}\left[\Par{u\Par{\vp_{\tau + \eta}} - u\Par{\vp_{\tau}} }\Par{g_{\tau + \eta}\Par{\vp_{\tau + \eta}} - g_\tau\Par{\vp_{\tau}}}\mid \vp_\tau = \vz\right]
\end{align*}
and combining the above two displays with the definition \eqref{eq:cdc_P} yields the claim.
\end{proof}

In the remainder of the section, we specialize Perspective~\ref{perspective:bridge} to the setting where $\mu = \delta_{\0_d}$, $\pi=\pi_0$, and $R = W$ is the Wiener path measure on $[0, 1]$ (so that $\mu = R_0$ agree). We recall the well-known identities for when $\vp_{0, 1} \sim W$ follows the heat flow $\dd \vp_\tau = \dd \vW_\tau$: 
\begin{equation}\label{eq:heat_gen}
\sfL_\tau^W u = \half \Delta u,\quad \sfGam_\tau^W(u, v_\tau) = \half \inprod{\nabla u}{\nabla v_\tau}.
\end{equation}
The second identity can be verified using an It\^o--Taylor expansion. We can now relate this special case of the Schr\"odinger bridge to stochastic localization.

\begin{theorem}\label{thm:renormal_bridge}
Let $P^\star$ solve \eqref{eq:schrodinger} when $\mu = \delta_{\0_d}$ and $R = W$ is the Wiener path measure on $[0, 1]$. Then if $\vv_\tau \sim P^\star_\tau$, the processes $\{\vc_t\}_{t \ge 0}$ in \eqref{eq:tilt_def} and $\{\vv_\tau\}_{\tau \in [0, 1]}$ satisfy $\vc_t = \frac 1 {1 - \tau} \vv_\tau$, where $t = \frac \tau {1 - \tau}$, and the induced $\pi_t$ in \eqref{eq:tilt_def} and $P^\star_{1 \mid \tau}(\cdot \mid \vv_\tau)$ are identical under this reparameterization.
\end{theorem}
\begin{proof}
We prove this by appealing to Theorem~\ref{thm:tilt_renormal}. In particular, we show that $P^\star_{1 \mid \tau}(\cdot \mid \vv_\tau)$ as defined in the theorem statement agrees with the induced fluctuation measure $\pi_\tau^{\vv_\tau}$ defined in \eqref{eq:fluctuation2}, and that $\vv_\tau$ defined in the theorem statement follows the SDE \eqref{eq:polchinski_sde}. 

To see the first claim, because $\mu = W_0 = \delta_{\0_d}$, the optimal (and only) coupling of $\mu$ and $\pi$ in \eqref{eq:schro-static} is the product coupling, so that $f \equiv 1$ and $g = \frac \pi {W_1}$ in \eqref{eq:gam-rep}. Now, the last part of \eqref{eq:pstar_forms} gives
\begin{align*}
P^\star_{1 \mid \tau}\Par{\vy \mid \vz} &= \frac{g(\vy)}{g_\tau(\vz)} W_{1 \mid \tau}(\vy \mid \vz) = \frac{\pi_0(\vy)}{g_\tau(\vz)} \frac{W_{1 \mid \tau}(\vy \mid \vz)}{W_1(\vy)} \\
&\propto \exp\Par{-\frac 1 {2(1 - \tau)}\norm{\vy - \vz}_2^2 + \half\norm{\vy}_2^2}\pi_0(\vy) \\
&\propto \exp\Par{\frac 1 {1 - \tau}\inprod{\vy}{\vz} - \frac \tau {2(1-\tau)}\norm{\vy}_2^2} \pi_0(\vy).
\end{align*}
This agrees with our definition of the fluctuation measure in \eqref{eq:fluctuation2}, when $\vv_\tau = \vz$.

To see the second claim, we instead show that $\sfL_\tau^{P^\star}$, the infinitesimal generator of $P^\star_\tau$, agrees with \eqref{eq:polchinski_gen_def}. If we can show this, Lemma~\ref{lem:polchinski_sde_derive} implies the SDE \eqref{eq:polchinski_sde} holds for $\vv_\tau$. By combining \eqref{eq:P-generator} and the specialized formulas for the heat flow \eqref{eq:heat_gen}, we have
\begin{align*}
L_\tau^{P^\star} u = \half \Delta u + \frac 1 {g_\tau} \inprod{\nabla g_\tau}{\nabla u} = \half \Delta u + \inprod{\nabla \log g_\tau}{\nabla u}.
\end{align*}
This is exactly the same as \eqref{eq:polchinski_gen_def}, as long as we can show $\nabla \log g_\tau = -\nabla V_\tau$. Applying \eqref{eq:factor_tau},
\[g_\tau(\vz) = \E_{\vp_{[0, 1]} \sim W}\Brack{g(\vp_1) \mid \vp_\tau = \vz} =  \E_{\vp_{[0, 1]} \sim W}\Brack{\frac{\pi(\vp)}{W_1(\vp)} \mid \vp_\tau = \vz}.\]
Further, $\frac \pi {W_1} = \frac 1 Z \exp(-V_1)$ as defined in Perspective~\ref{perspective:renorm} for some normalizing constant $Z$. Moreover, conditioned on $\vp_\tau = \vz$, we know that $\vp_1 \sim \Nor(\vz, (1 - \tau)\id_d)$ under $W$. Thus,
\begin{align*}g_\tau(\vz) &= \frac 1 Z \E_{\vp_1 \sim \Nor(\vz, (1 - \tau)\id_d)}\Brack{\exp\Par{-V_1(\vp_1)}} \\
&= \frac 1 Z \E_{\vg \sim \Nor(\0_d, (1 - \tau)\id_d)}\Brack{\exp\Par{-V_1(\vz + \vg)}} = \frac 1 Z \exp\Par{-V_\tau(\vz)}, \end{align*}
at which point it is clear that $\nabla \log g_\tau = -\nabla V_\tau$ pointwise.
\end{proof}

The special case of the Schr\"odinger bridge to a Wiener reference measure in Theorem~\ref{thm:renormal_bridge} is extremely well-studied. In particular, the transition kernels in \eqref{eq:pstar_forms} are an instance of Doob's $h$-transform \cite{Doob01}, and the resulting optimal path measure $P^\star$ is induced by the \emph{F\"ollmer drift} \cite{Foll1985, Foll1986}; see further discussion in Section~\ref{sec:dynamic}, as well as \cite{Lehec13} for applications of this construction.

We include a remark demonstrating the broader applicability of these techniques.

\begin{remark}
    Consider the path measure $R$, defined as follows: $R_0 = \delta_{\0_d}$, and $\vv_\tau \sim R_\tau$ satisfies the SDE $\D \vv_\tau=\vm_\tau(\vv_\tau)\,\D \tau+\msig_\tau(\vv_\tau)\,\D \vW_\tau$.  In general, it is difficult to calculate the transition densities for $P^\star$, the optimizer to \eqref{eq:schrodinger} with the stated $R$ and $\mu = \delta_{\0_d}$.  Nevertheless, using the techniques from this section, we can derive an SDE for $\vv_{[0, 1]} \sim P^\star$: defining $g \defeq \frac{P^\star_1}{R_1}$ and $g_\tau$ as in \eqref{eq:factor_tau},
\[\D\vv_\tau=\left(\vm_\tau(\vv_\tau)+\msig_\tau(\vv_\tau)\msig_\tau(\vv_\tau)^\top \nabla \log g_\tau(\vv_\tau)\right)\, \D \tau+\msig_\tau(\vv_\tau)\,\D \vW_\tau.\]
    For a complete derivation, see Section 3.3 of \cite{Chewi25}.
\end{remark}

\section{Dynamic Schr\"odinger bridge}\label{sec:dynamic}

We present the first of two alternative perspectives on the static Schr\"odinger bridge problem: a dynamic reformulation framed as computing an optimal drift with a target end measure. 

In this section, let $L^2([0, 1]; \R^d) \defeq \{\mathbf u: [0, 1] \to \R^d, \int_0^1 \norm{\mathbf u_\tau}_2^2\,\dd\tau < \infty\}$, let $\{\vW_\tau\}_{\tau \in [0, 1]}$ be a Wiener process, and let $\{\msc F_\tau\}_{\tau \in [0, 1]}$ be the filtration generated by $\{\vW_\tau\}_{\tau \in [0, 1]}$.

\begin{tcolorbox}[colback=blue!10, colframe=blue!50!black, boxrule=0.5pt, arc=2mm]
\begin{perspective}\label{perspective:dynamic-eot}
    Let $\mu, \pi \in \calP(\R^d)$, and define 
    \[\calU \defeq \{\vu \in L^2([0, 1], \R^d)\mid \vu \text{ is adapted to } \{\msc F_\tau\}_{\tau \in [0, 1]}\}.\]
    The induced \emph{dynamic Schr\"odinger bridge} problem is:
\begin{equation}\label{eq:dynamic-version-1}\tag{$\msf{DSB}$}
\begin{gathered}
    \inf_{\mathbf u \in \mc U} \E\Bigl[\frac{1}{2} \int_0^1 \norm{\mathbf{u_\tau}}_2^2 \, \D \tau \Bigr], \\ 
    \text{such that } \D \mathbf x_\tau = \mathbf u_\tau \, \D \tau + \D \vW_\tau ,\quad \operatorname{law}(\mathbf x_0)=\mu, \quad \operatorname{law}(\mathbf x_1)=\pi.
\end{gathered}
\end{equation}
\end{perspective}
\end{tcolorbox}

We make use of Girsanov's theorem, which is a classic change-of-measure principle on the space $L^2([0, 1]; \R^d)$. We note that the specialization of Girsanov's theorem to determinsitic drifts was already introduced as the Cameron-Martin theorem (Lemma~\ref{lem:cameron_martin}).

\begin{lemma}[{Girsanov's theorem; adapted from~\cite[Theorem 3.2.8, Theorem 4.4.1]{Chewi25}}]\label{lem:girsanov}
    Suppose $\{\mathbf x_\tau\}_{\tau \in [0, 1]}$ is simultaneously driven by the following SDEs, for $\{\mathbf u_\tau\}_{\tau \in [0, 1]}$ adapted to $\{\msc F_\tau\}_{\tau \in [0, 1]}$:
    \begin{align*}
        \D \mathbf x_\tau = \mathbf u_\tau \, \D \tau + \D \vW_\tau, \qquad \D \mbf x_\tau = \D \widetilde \vW_\tau,
    \end{align*}
    with initial condition $\mathbf x_0 = \mathbf a$. Let $P_{\mbf a}$ be the probability measure under which $\{\vW_\tau\}_{\tau \in [0, 1]}$ is a standard Brownian motion, and similarly let $W_{\mbf a}$ be the probability measure under which $\{\widetilde \vW_\tau\}_{\tau \in [0, 1]}$ is Brownian. 
    Then,\footnote{Formally, Girsanov's theorem requires that a technical condition called \emph{Novikov's condition}. This condition can be ignored for our application, as for the $\mathrm{KL}$ divergence, it suffices to apply a standard localization argument combined with lower semicontinuity. See~\cite[Section 5.2]{ChenCLLSZ23} for a similar argument.} $P_{\mbf a}$ is absolutely continuous with respect to $W_{\mathbf a} \in \mc P([0, 1] \times \R^d)$, and
    \begin{align*}
       \E_{P_{\mbf a}}\Bigl[ \log \frac{P_{\mbf a}}{ W_{\mbf a}}\Bigr] = \frac{1}{2} \int_0^1 \E_{P_{\mbf a}} \Bigl[\norm{\mbf u_\tau}_2^2\Bigr] \, \D \tau.
    \end{align*}
\end{lemma}

As a last helper definition, when $\mu$ is a probability measure on $\R^d$, we define the \emph{generalized Wiener measure} $W_\mu \in \calP([0, 1] \times \R^d)$ as $W_\mu \defeq \E_{\va \sim \mu} W_{\va}$. Equivalently, $\vp_{[0, 1]} \sim W_\mu$ follows
\[\vp_0 \sim \mu,\quad \vp_{(0, 1] \mid 0} \sim W_{\vp_0}.\]

\begin{theorem}\label{thm:dynamic_static}
    The path measure $P = \operatorname{law}(\{\mathbf x_\tau\}_{\tau \in [0, 1]})$ induced by the optimal solution to~\eqref{eq:dynamic-version-1} is the optimal solution to~\eqref{eq:schrodinger} with $R = W_\mu$.
\end{theorem}
\begin{proof}
    Starting from~\eqref{eq:schrodinger} and using the chain rule for $\mathrm{KL}$, we have
    \begin{align*}
        \KL{P}{R_\mu} &= \KL{P_0}{\mu} + \E_{\mathbf x_0 \sim \mu}\Brack{ \KL{P_{(0, 1] \mid 0}(\cdot \mid \mathbf x_0)}{W_{\mathbf x_0}}} \\
        &= \E_{\mathbf x_0 \sim \mu}\Brack{ \KL{P_{(0, 1] \mid 0}(\cdot \mid \mathbf x_0)}{W_{\mathbf x_0}}} \le \frac{1}{2} \E_{P} \Bigl[\int_0^1 \norm{\mathbf u_\tau}_2^2 \, \D t\Bigr].
    \end{align*}
    The second line used $P_0 = \mu$. We then applied Lemma~\ref{lem:girsanov}: $\{\vx_\tau\}_{\tau \in [0, 1]}$ is adapted to the underlying Brownian path measures, so the KL divergence is bounded via the data processing inequality.
\end{proof}

When $\mu = \delta_{\0_d}$ in Theorem~\ref{thm:dynamic_static}, the measure $W_\mu$ is just the standard Wiener measure. In this special case (the setting of Theorem~\ref{thm:renormal_bridge}), the optimal drift $\{\vu_\tau\}_{\tau \in [0, 1]}$ in \eqref{eq:dynamic-version-1} is called the F\"ollmer drift. 
The full discussion around the equivalence of these two areas (and generalizations thereof) contains many subtle points and is a fertile area of mathematical research.

%% file: sections/eot.tex
\section{Entropic optimal transport}\label{sec:eot}

We conclude with a second alternative perspective on the static Schr\"odinger bridge problem: a connection to entropic optimal transport, following \cite{Peyre2019l, ChewiNilesWeed2025}. 

\begin{tcolorbox}[colback=blue!10, colframe=blue!50!black, boxrule=0.5pt, arc=2mm]
\begin{perspective}\label{perspective:eot}
Let $\mu, \pi \in \calP(\R^d)$, and let $\varepsilon > 0$. The induced \emph{entropic optimal transport} problem is:
\begin{align}\label{eq:EOT}\tag{$\mathsf{EOT}$}
    \underset{{\gamma \in \Gamma(\mu, \pi)}}{\min} \biggl\{\int \half\norm{\mbf x - \mbf y}_2^2 \,\dd \gamma(\mbf x, \mbf y) + \varepsilon \KL{\gamma}{\mu \otimes \pi} \biggr\}\,.
\end{align}
\end{perspective}
\end{tcolorbox}
Note that the solution to \eqref{eq:EOT} exists and is unique, by strict convexity of $\text{KL}$ in its first argument, and compactness of the set $\Gamma(\mu, \pi)$ (defined in Section~\ref{sec:schrodinger}), the set of all couplings of $\mu$ and $\pi$. Here, $\mu \otimes \pi$ refers to the trivial coupling (product measure) between $\mu$ and $\pi$.

When $\eps \searrow 0$, \eqref{eq:EOT} recovers the classical optimal transport problem~\cite{ChewiNilesWeed2025}. One motivation for studying \eqref{eq:EOT} is that while solving classical optimal transport can often be computationally intractable,~\eqref{eq:EOT} can be efficiently solved via methods such as Sinkhorn's algorithm~\cite{Sink1967}.

We now relate \eqref{eq:EOT} to \eqref{eq:schrodinger} in the case of generalized Wiener reference measures.

\begin{theorem}
    The optimization problems in~\eqref{eq:EOT} (with $\varepsilon=1$) and~\eqref{eq:schrodinger} (with $R = W_\mu$) are the same, up to an additive constant shift (depending only on $\mu, \nu$).
\end{theorem}
\begin{proof}
We start from~\eqref{eq:schrodinger}. As the initial condition $R_0 = \mu$ is fixed,
\begin{align*}
    \underset{{\gamma \in \Gamma(\mu, \pi)}}{\min} \;  \KL{\gamma}{R_{0, 1}}
    =&\underset{{\gamma \in \Gamma(\mu, \pi)}}{\min}\;  \mathbb{E}_{\mathbf{x}\sim \mu}[\KL{\gamma(\mathbf{y}\mid \mathbf{x})}{R_{1\mid 0}(\mathbf{y}\mid \mathbf{x})}]+\KL{\mu(\mathbf{x})}{R_0(\mathbf{x})}\\
    =&\underset{{\gamma \in \Gamma(\mu, \pi)}}{\min}\; \int -\log R_{0, 1}(\mbf{x}, \mbf{y})\gamma(\mbf{x}, \mbf{y})\, \D \mbf{x}\, \D\mbf{y}+\int \gamma(\mbf{x}, \mbf{y})\log \gamma(\mbf{x}, \mbf{y})\, \D\mbf{x}\, \D\mbf{y}
\end{align*}
The first line used Lemma~\ref{lem:endpoints_suffice}.
Now, as $R = W_\mu$ is a generalized Wiener measure, using the explicit form of $\log R_{0,1}(\vx, \vy) = -\half \norm{\vx - \vy}_2^2 + C$ for a constant $C$, and 
\[\int \gamma(\mbf{x}, \mbf{y})\log \gamma(\mbf{x}, \mbf{y})\, \D\mbf{x}\, \D\mbf{y} - \KL{\gamma}{\mu \otimes \pi} = \int \mu(\vx)\log \mu(\vx)\,\dd\vx + \int \nu(\vy)\log \nu(\vy) \,\dd\vy,\]
where the right-hand side is independent of $\gamma$, reduces the problem to~\eqref{eq:EOT}.
\end{proof}

%% file: sections/anisotropic.tex
\section{Anisotropic stochastic localization}\label{app:anisotropic}

In this section we give a more general variant of \eqref{eq:tilt_def} based on its presentation in \cite{EldanMZ20}, that allows for non-identity ``control matrices'' to drive the process.

Let $\{\mathbf{C}_t\}_{t \ge 0}$ be an $\R^{d \times d}$-valued process that is adapted to the filtration $\{\msc F_t\}_{t \ge 0}$ generated by a Wiener process $\{\vW_t\}_{t \ge 0}$, and let $\pi_0 \in \calP(\R^d)$ be such that $\vm_0 \defeq \E_{\vx \sim \pi_0}[\vx]$ exists. We define the following tilt-valued process, and $\norm{\vx}_{\msig}^2 \defeq \vx^\top \msig \vx$ for positive semidefinite $\msig$:
\begin{equation}\label{eq:tilt_def_anisotropic}\tag{$\msf{ASL}$-$\msf{I}$}
\begin{gathered}\vc_0 = \0_d,\quad \dd \vc_t = \mmc_t \mmc_t^\top \mathbf{m}_t \,\dd t + \mathbf{C}_t\,\dd \vW_t,\\
\text{where } \mathbf{m}_t \defeq \E_{\vx \sim \pi_t}[\vx],\quad \pi_t(\vx) \propto \exp\Par{\inprod{\vc_t}{\vx} - \frac{1}{2} \norm{\vx}_{\bs\Sigma_t}^2}\pi_0(\vx),\quad \dd\msig_t = \mmc_t \mmc_t^\top \,\dd t.\end{gathered}
\end{equation}
Note that when $\mmc_t = \id_d$ for all $t \ge 0$, \eqref{eq:tilt_def_anisotropic} reduces to \eqref{eq:tilt_def}. Similarly to the duality between Perspectives~\ref{perspective:tilt} and~\ref{perspective:density}, there is an equivalent anisotropic measure-valued process to \eqref{eq:tilt_def_anisotropic}:
\begin{equation}\label{eq:measure_def_anisotropic}\tag{$\msf{ASL}$-$\msf{II}$}
\dd \pi_t(\vx) = \inprod{\vx - \vm_t}{\mmc_t \,\dd \vW_t} \pi_t(\vx), \text{ where } \vm_t \defeq \E_{\vx \sim \pi_t}[\vx].
\end{equation}

\begin{theorem}
    The dynamics of $\pi_t$ given by \eqref{eq:tilt_def_anisotropic} and \eqref{eq:measure_def_anisotropic} are the same.
\end{theorem}
\begin{proof}
The direction that starts from \eqref{eq:measure_def_anisotropic} and derives \eqref{eq:tilt_def_anisotropic} is a straightforward generalization of the proofs of Lemma~\ref{lem:pit_density} and Theorem~\ref{thm:tilt_density}. For convenience to the reader, in this proof we provide the opposite direction for the more general processes \eqref{eq:tilt_def_anisotropic}, \eqref{eq:measure_def_anisotropic}.

Let us begin with \eqref{eq:tilt_def_anisotropic}. Define the normalization constant and its renormalizing potential:
\[Z_t \defeq \int \exp\Par{\inprod{\vc_t}{\vx} - \half\norm{\vx}_{\msig_t}^2}\pi_0(\vx)\,\dd\vx,\quad h_t(\vx) \defeq \inprod{\vc_t}{\vx} - \half\norm{\vx}_{\msig_t}^2.\]
Then by applying It\^o's lemma (Lemma~\ref{lem:ito}) to $Z_t = \int \exp(h_t(\vx)) \pi_0(\vx) \,\dd\vx$, where we note that the diffusion term in $\dd h_t(\vx)$ is $\inprod{\mmc_t^\top\vx}{\dd\vW_t}$, we have
\begin{align*}
\dd Z_t &= \int \dd \exp(h_t(\vx))\pi_0(\vx) \, \dd \vx\\
&= \int \Par{-\half \norm{\vx}_{\mmc_t \mmc_t^\top}^2 \,\dd t + \half \norm{\vx}_{\mmc_t \mmc_t^\top}^2 \,\dd t} \exp\Par{h_t(\vx)}\pi_0(\vx) \,\dd\vx \\
&\qquad+ \int \inprod{\vx}{\dd \vc_t}\exp\Par{h_t(\vx)}\pi_0(\vx)\,\dd\vx = \int \inprod{\vx}{\dd \vc_t}\exp\Par{h_t(\vx)}\pi_0(\vx)\,\dd\vx.
\end{align*}
Thus, applying It\^o's lemma again,
\begin{align*}
\dd \log Z_t &= \frac 1 {Z_t}\Par{\int \inprod{\vx}{\dd \vc_t}\exp\Par{h_t(\vx)}\pi_0(\vx)\,\dd\vx} - \frac 1 {2Z_t^2}\norm{\int \vx \exp\Par{h_t(\vx)}\pi_0(\vx)\,\dd\vx}_{\mmc_t\mmc_t^\top}^2 \dd t \\
&= \inprod{\vm_t}{\dd \vc_t} - \half\norm{\vm_t}_{\mmc_t \mmc_t^\top}^2 \,\dd t.
\end{align*}
Now because $\log \pi_t(\vx) = \inprod{\vc_t}{\vx} - \half\norm{\vx}_{\msig_t}^2 + \log(\pi_0(\vx)) - \log Z_t$,
\begin{align*}
\dd \log \pi_t(\vx) &= \inprod{\vx}{\dd \vc_t} - \half\norm{\vx}_{\mmc_t \mmc_t^\top}^2 \,\dd t - \dd \log Z_t \\
&= \inprod{\vx - \vm_t}{\dd \vc_t} + \Par{\half\norm{\vm_t}_{\mmc_t\mmc_t}^2 - \half\norm{\vx}_{\mmc_t\mmc_t}^2} \dd t \\
&= \inprod{\vx - \vm_t}{\dd \vc_t - \mmc_t\mmc_t^\top \vm_t \,\dd t} - \half\norm{\vx - \vm_t}_{\mmc_t \mmc_t^\top}^2 \,\dd t \\
&= \inprod{\vx - \vm_t}{\mmc_t \,\dd \vW_t} - \half\norm{\vx - \vm_t}_{\mmc_t \mmc_t^\top}^2 \,\dd t.
\end{align*}
Applying It\^o's lemma one final time then yields the SDE \eqref{eq:measure_def_anisotropic} as claimed.
\end{proof}

%% file: sections/rgo.tex
\section{Restricted Gaussian dynamics}\label{app:rgo}

\subsection{The perspective from entropic stability}
In this appendix, we discuss an algorithmic application of stochastic localization: the restricted Gaussian dynamics \cite{LeeST21}. Our presentation follows the \emph{localization schemes} analysis framework of \cite{ChenE22}, and in particular, this appendix serves to replicate the results of \cite{ChenE22} that are specialized to the restricted Gaussian dynamics. We organize this appendix as follows.

\begin{enumerate}[label=(\arabic*)]
    \item Appendix~\ref{app:schemes} introduces the concept of a localization process and describes its induced localization dynamics Markov chain. Specializing this framework to the stochastic localization process \eqref{eq:tilt_def} gives rise to the restricted Gaussian dynamics \eqref{eq:rgd}.

    \item Appendix~\ref{app:stability} develops tools for proving that localization processes conserve entropy. We describe the notion of entropic stability (Definition~\ref{def:ent_stable}) from \cite{ChenE22}, give a sufficient condition for this notion (Lemma~\ref{lem:tilt-stability}), and show how it implies entropy conservation (Lemma~\ref{lem:conservation}).

    \item Appendix~\ref{app:lsi} shows how conservation of entropy, in the sense of Lemma~\ref{lem:conservation}, implies a modified log-Sobolev inequality. We then apply this machinery to the restricted Gaussian dynamics \eqref{eq:rgd} to derive a mixing time bound for strongly log-concave stationary measures.
\end{enumerate}

\subsubsection{Localization schemes}\label{app:schemes}

The restricted Gaussian dynamics is an instance of the localization schemes framework of \cite{ChenE22}, specialized to the stochastic localization process \eqref{eq:tilt_def}. We first require a general definition.

\begin{definition}\label{def:loc-proc}
   A \emph{localization process} is a measure-valued stochastic process $\{\pi_t\}_{t \geq 0}$ on a state space $\Omega$, which satisfies the following conditions.
    \begin{enumerate}[label=(\arabic*)]
        \item $\pi_t$ is a probability measure over $\Omega$ for all $t$ almost surely.

        \item For all $A \subseteq \Omega$, $t \mapsto \pi_t(A)$ is a martingale.

        \item For all $A \subseteq \Omega$, $\lim_{t \to \infty} \pi_t(A) \in \{0, 1\}$ almost surely.
    \end{enumerate}
    A \emph{localization scheme} maps $\pi$, a measure over $\Omega$, to a localization process $\{\pi_t\}_{t \geq 0}$ with initial condition $\pi_0 = \pi$.     Lastly, the \emph{localization dynamics} associated with the localization process and some fixed time $T > 0$ is the Markov chain with transition kernel defined as
    \begin{equation}\label{eq:kernel_loc}\begin{aligned}
        \msf P_{T}^\pi(A \mid \omega) = \E \Brack{\frac{\pi_T(\omega) \pi_T(A)}{\pi(\omega)} },\text{ for all } \omega \in \Omega,\; A \subseteq \Omega.
    \end{aligned}\end{equation}
\end{definition}

We recall a basic fact about localization dynamics.
\begin{lemma}\label{lem:pi_stationary}
For any $T > 0$, $\pi$ is stationary for the transition kernel $\sfP_T^\pi$.
\end{lemma}
\begin{proof}
By Fubini's theorem,
\begin{align*}
    \int_\Omega \msf P_T^\pi(A \mid \omega) \pi(\omega)\,\dd\omega &= \int_\Omega \E \Brack{\frac{\pi_T(\omega) \pi_T(A)}{\pi(\omega)} } \pi(\omega)\,\dd\omega =  \E \Brack{\int_\Omega \frac{\pi_T(\omega) \pi_T(A)}{\pi(\omega)}  \pi(\omega)\,\dd\omega} \\
    &= \E\Brack{ \pi_T(A)\int_\Omega \pi_T(\omega)\,\dd\omega } = \pi(A),
\end{align*}
where we used properties (1) and (2) in Definition~\ref{def:loc-proc}.
\end{proof}

We have seen a localization scheme for measures $\pi$ on $\Omega \defeq \R^d$: the stochastic localization process \eqref{eq:tilt_def}, \eqref{eq:density_def}. Condition (1) of Definition~\ref{def:loc-proc} is Lemma~\ref{lem:pit_density}, condition (2) is immediate from the dynamics \eqref{eq:density_def}, and condition (3) follows from the form of $\pi_t$ in \eqref{eq:tilt_def}.

We next define the restricted Gaussian dynamics \cite{LeeST21}, the focus of the rest of the section. For a measure $\pi$ on $\R^d$, the restricted Gaussian dynamics have transitions $\vx \to \vx'$ given by
\begin{equation}\label{eq:rgd}\tag{$\msf{RGD}$}
\begin{aligned}
\vy \sim \Nor(\vx, \eta \id_d),\quad \text{Law}(\vx') \propto \exp\Par{-\frac 1{2\eta}\norm{\vx' - \vy}_2^2}\pi(\vx'),
\end{aligned}
\end{equation}
for some $\eta > 0$. We observe that \eqref{eq:rgd} is an instance of localization dynamics, following \cite{ChenE22}. 
\begin{theorem}\label{thm:rgd_is_localization}
The Markov chain with transition kernel defined by \eqref{eq:rgd} is the same as the localization dynamics associated with the localization scheme given by \eqref{eq:tilt_def}, \eqref{eq:density_def}, with $T = \frac 1 \eta$.
\end{theorem}
\begin{proof}
Theorem~\ref{thm:tilt_posterior} shows that $\pi_T$ in \eqref{eq:tilt_def}, \eqref{eq:density_def} can be equivalently described as follows: it is the posterior distribution $\pi(\cdot \mid \vc)$, where $\vw \sim \pi$ and $\vc \sim \Nor(T\vw, T\id_d)$. For notational convenience in this proof, let $\rho$ denote the joint distribution of $(\vw, \vc)$, let $\mu$ denote the $\vc$-marginal of $\rho$, and let $\pi(\cdot \mid \vc)$, $\mu(\cdot \mid \vw)$ denote the conditional laws of one marginal of $\rho$ given the other.

We now compute the transition density from $\vx$ to $\vx'$ according to \eqref{eq:kernel_loc}:
\begin{align*}
\sfP_T^\pi(\vx' \mid \vx) &= \E_{(\vw, \vc) \sim \rho}\Brack{\frac{\pi(\vx \mid \vc)\pi(\vx'\mid \vc)}{\pi(\vx)}} = \E_{\vc \sim \mu}\Brack{\frac{\rho(\vx, \vc)\pi(\vx'\mid \vc)}{\mu(\vc)\pi(\vx)}} \\
&= \int \frac{\rho(\vx, \vc)\pi(\vx'\mid \vc)}{\pi(\vx)} \,\dd \vc = \int \mu(\vc \mid \vx)\pi(\vx' \mid \vc)\,\dd\vc.
\end{align*}
Thus, an equivalent way to perform this transition is to sample $\vc \sim \mu(\cdot \mid \vx) = \Nor(T\vx, T\id_d)$ and then $\vx' \sim \pi(\cdot \mid \vc)$.  Under the variable transformation $\vy \defeq \frac 1 T \vc$ and $\eta \defeq \frac 1 T$, it is equivalent to sample $\vy \sim \Nor(\vx, \eta \id_d)$ and then $\vx' \sim \pi(\cdot \mid T\vy)$. By Bayes' theorem,
\begin{align*}
\pi(\vx' \mid T\vy) &\propto \rho(\vx', T\vy) \propto \exp\Par{-\frac{1}{2T}\norm{T\vx' - T\vy}_2^2}\pi(\vx') = \exp\Par{-\frac{1}{2\eta}\norm{\vx'-\vy}_2^2}\pi(\vx').
\end{align*}
This equivalent sampling process is exactly as described in \eqref{eq:rgd}.
\end{proof}

The rest of the section establishes a mixing time estimate on the dynamics \eqref{eq:rgd}.

\subsubsection{Entropy conservation from entropic stability}\label{app:stability}

In this section, we use the notion of \emph{entropic stability} to establish that entropy is conserved, in a precise sense, along the localization process $\{\pi_t\}_{t \ge 0}$ given by \eqref{eq:tilt_def}.
Before proceeding, we introduce a suite of definitions, which are adapted from~\cite[Section 3.2.1]{ChenE22}.
\begin{definition}[Barycentres]
    Denote the \emph{barycentre} $\mbf b: \mc P(\R^d) \to \R^d$ of a measure as $\mbf b(\pi) = \E_\pi[\mbf x]$.
\end{definition}

\begin{definition}[Exponential tilt]
    Define the \emph{exponential tilt operator} $\mc T: \R^d \times \mc P(\R^d) \to \mc P(\R^d)$ as:
    \begin{align*}
        \mc T_{\mbf y} \pi(\mbf x) \propto \exp(\langle\mbf y, \mbf x\rangle) \pi(\mbf x),
    \end{align*}
    whenever $\exp(\inprod{\vy}{\cdot})\pi$ is integrable over $\R^d$.
\end{definition}

\begin{definition}[Entropic stability]\label{def:ent_stable}
We say that    $\pi \in \mc P(\R^n)$ is $\alpha$-entropically stable if
    \begin{align*}
        \half\norm{\mbf b(\mc T_{\mbf y} \pi) - \mbf b(\pi)}_2^2 \leq \alpha \KL{\mc T_{\mbf y} \pi}{ \pi}, \text{ for all } \mbf y \in \R^d.
    \end{align*}
\end{definition}

We remark that Definition~\ref{def:ent_stable} is used in greater generality in \cite{ChenE22}, with alternative distance functions on the left-hand side. For example, a variant where $\half\norm{\cdot - \cdot}_2^2$ is replaced with a modified KL divergence is used to study the Glauber dynamics on the hypercube.

Let us record an auxiliary lemma. In its statement, we denote the covariance of $\pi \in \calP(\R^d)$ by
\[\Cov(\pi) \defeq \E_{\vx \sim \pi}\Brack{(\vx - \vb(\pi))(\vx - \vb(\pi))^\top}. \]

\begin{lemma}[{\cite[Lemma 1]{BubeckE19}}]\label{lem:kl-deriv}
    Let $\pi \in \mc P(\R^d)$ be a measure such that $\Cov(\pi)$ exists and is everywhere invertible, and that the cumulant generating function (a.k.a.\ log-Laplace transform)
    \begin{align*}
        \chi(\vy) = \log \E_{\pi}[\exp(\langle \vy, \cdot \rangle)]
    \end{align*}
    exists.
    There exists a unique function $\vy: \calK \coloneqq \operatorname{int}\bigl(\operatorname{conv} (\operatorname{supp}(\pi))\bigr) \to \R^d$ such that for all $\vx \in \calK$,
    \begin{align*}
        \mbf b(\mc T_{\vy(\vx)} \pi) = \mbf x,
    \end{align*}
    and furthermore if $\Phi(\vx) \defeq \KL{\mc T_{\vy(\vx)} \pi}{\pi}$, $\Phi$ and $\chi$ are convex conjugates, and on $\calK$,
    \begin{align*}
        \nabla \Phi (\vx) = \vy(\vx), \quad \nabla^2 \Phi(\mbf x) =  \Cov(\mc T_{\vy(\vx)} \pi)^{-1}.
    \end{align*}
\end{lemma}
\begin{proof}
We provide a brief sketch here, deferring more details to \cite{BubeckE19}. The key observation is 
\begin{align*}
\nabla \chi(\vy) = \vb(\tilt_{\vy}\pi),\quad \nabla^2 \chi(\vy) = \Cov\Par{\tilt_{\vy}\pi}.  
\end{align*}
which follows from a direct computation (and because the mean and covariance are the first two cumulants). Now, for any $\vy \in \R^d$, first-order optimality shows
\begin{align*}
\vb(\tilt_{\vy}\pi) \in \arg\max_{\vx \in \R^d} \inprod{\vx}{\vy} - \chi(\vy),
\end{align*}
and we can check that
\begin{align*}
\chi^*(\vb(\tilt_{\vy}\pi)) &= \inprod{\vy}{\vb(\tilt_{\vy}\pi)} - \chi(\vy) \\
&= \int \inprod{\vy}{\vx} \tilt_{\vy} \pi(\vx) \,\dd\vx - \chi(\vy) \\
 &= \int \tilt_{\vy} \pi(\vx) \log\Par{\frac{\tilt_{\vy} \pi(\vx)}{\pi(\vx)}}\,\dd\vx = \KL{\tilt_{\vy} \pi}{\pi}.
\end{align*}
This proves the conjugacy of $\chi$ and $\Phi$. The fact that $\nabla \Phi$ is the inverse mapping of $\vy \to \vb(\tilt_{\vy} \pi) = \nabla \chi(\vy)$, and that $\nabla^2 \Phi(\vx)$ is the inverse of $\nabla^2 \chi(\vy(\vx))$, then follow from standard properties of convex conjugates. We defer the proof of uniqueness of $\vy$ to \cite{BubeckE19}.
\end{proof}

Using Lemma~\ref{lem:kl-deriv}, we give a sufficient condition for $\alpha$-entropic stability.

\begin{lemma}[{\cite[Lemma 40]{ChenE22}}]\label{lem:tilt-stability}
If $\Cov(\calT_{\vy} \pi) \preceq \alpha \id_d$ for all $\vy \in \R^d$, 
    $\pi$ is $\alpha$-entropically stable.
\end{lemma}
\begin{proof}
    Let $\vy(\vx)$ and $\Phi(\vx) \defeq \KL{\mc T_{\vy(\mbf x)}\pi}{\pi}$ be as in Lemma~\ref{lem:kl-deriv}. Then,
    \begin{align*}
        \nabla^2 \Phi(\mbf x) = \Cov(\mc T_{\vy(\mbf x)} \nu)^{-1} \succeq \frac{1}{\alpha} \mbf I_d\,,
    \end{align*}
    where the first equality is Lemma~\ref{lem:kl-deriv}, whereas the second is by assumption. Now, define $\widetilde \Phi(\mbf x) = \frac{1}{2\alpha} \norm{\mbf x - \mbf b(\pi)}_2^2$. The lemma statement asks to show $\Phi \ge \wt \Phi$ pointwise. Observe that
    \[\nabla \Phi(\vb(\pi)) = \nabla \widetilde \Phi(\mbf b(\pi)) = \0_d,\quad \nabla^2 \Phi \succeq \nabla^2 \wt \Phi = \frac 1 \alpha \id_d.\]
    Thus, as $\Phi$ and $\wt \Phi$ agree up to first order around $\mbf b(\pi)$, the second order term in $\Phi$ dominates that of $\wt \Phi$, and $\wt \Phi$ is a quadratic, it follows that $\Phi(\mbf x) \geq \wt \Phi(\mbf x)$ everywhere, since
    \begin{align*}
        \Phi(\mbf x) - \wt \Phi(\mbf x) = \int_0^1 \int_0^t (\mbf x - \mbf b(\pi))^\top \nabla^2 (\Phi - \wt \Phi)(\mbf b(\pi) + s(\mbf x - \mbf b(\pi))) \cdot (\mbf x - \mbf b(\pi)) \, \D s \, \D t \geq 0.
    \end{align*}
\end{proof}

We next use the \emph{maximum entropy principle} to piggyback off Definition~\ref{def:ent_stable} and show that entropic stability holds not just for exponential tilts, but all absolutely continuous measures.

\begin{lemma}\label{lem:entropic-stable-equiv}
    Suppose $\pi$ is $\alpha$-entropically stable, and that $\mbf b(\mc T_{\mbf y} \pi)$ exists for every $\mbf y \in \R^d$. For every measure $\nu$ absolutely continuous with respect to $\pi$ where $\vb(\nu)$ exists, we have
    \begin{align*}
        \half\norm{\mbf b(\nu) - \mbf b(\pi)}_2^2 \leq \alpha \KL{\nu}{\pi}.
    \end{align*}
\end{lemma}
\begin{proof}
We recall the maximum entropy principle: among all absolutely continuous measures to $\pi$ with a given mean, the one closest to $\pi$ in KL divergence is an exponential tilt of $\pi$. 
    First, define 
    \[\nu_{\mbf w}^* \deq \arg\min_{\substack{\nu \in \mc P(\R^d) \\ \mbf b(\nu) = \mbf w}}\KL{\nu}{\pi},\]
    and note that by standard variational calculus, defining the Lagrangian $\mc L(\nu, \lambda) = \KL{\nu}{\pi} + \langle \bs \lambda_1, \mbf b(\nu) - \mbf w \rangle + \lambda_2 (\E_{\nu}[1]- 1)$ for Lagrange multipliers $\bs \lambda_1 \in \R^d, \lambda_2 \in \R$, we have
    \begin{align*}
        \log \frac{\nu_{\mbf w}^*}{\pi}(\mbf x) + \lambda_2 + \langle \bs \lambda_1, \mbf x \rangle = 0\,.
    \end{align*}
    This tells us that
    \begin{align}\label{eq:tilt-optimality}
        \D \nu_{\vw}^*(\mbf x) \propto \exp(-\langle \bs \lambda_1, \mbf x \rangle) \D \pi(\mbf x),
    \end{align}
    for some vector $\bs \lambda_1$ fixing the barycentre, or in other words, that $\nu^\star_{\vw}$ is a linear tilt of $\pi$.

    Then, we note that for any function $g: \R^d \to \R_+$,
    \[
        \inf_{\mbf y \in \R^d} \frac{\KL{\mc T_{\mbf y} \pi}{\pi}}{g(\mbf b(\mc T_{\mbf y} \pi))}  \overset{\text{(i)}}{=} \inf_{\mbf w \in \R^d} \inf_{\substack{\nu \in \mc P(\R^d) \\ \mbf b(\nu) = \mbf w}} \frac{\KL{\nu}{\pi}}{g(\mbf w)} =\inf_{\nu \in \mc P(\R^d)} \frac{\KL{\nu}{\pi}}{g(\mbf b(\nu))},
    \]
    where (i) uses~\eqref{eq:tilt-optimality}.\footnote{Here, we need to restrict to measures with finite barycentres.}
    Applying this to $g(\vx) \defeq \half\norm{\mbf x - \mbf b(\pi)}_2^2$ concludes the proof.
\end{proof}

We are now ready to state the main result of this section, which shows that on average over an entropically stable localization scheme, the entropy (defined below) is conserved:
\begin{align*}
\Ent_\pi[f] \defeq \E_{\pi}[f \log f] - \E_\pi[f] \log\E_\pi[f], \text{ for } \pi \in \calP(\R^d),\; f: \R^d \to \R_{+}.
\end{align*}
We will use the following identity: for $\nu \propto \pi f$,
\begin{equation}\label{eq:ent_kl}
\KL{\nu}{\pi} = \E_\nu\Brack{\log \frac{f}{\E_\pi f}} = \E_\nu[\log f] - \log \E_\pi[f] = \frac{\Ent_\pi[f]}{\E_\pi[f]}.
\end{equation}

\begin{lemma}[{\cite[Proposition 39]{ChenE22}}]\label{lem:conservation}
Let $\{\pi_t\}_{t \ge 0}$ be a localization process for $\pi \in \calP(\R^d)$.
    Fix $T > 0$, and suppose that $\pi_t$ is almost surely $\alpha_t$-entropically stable for all $t \in [0, T]$. Then, 
    \begin{align*}
        \E\Brack{\operatorname{Ent}_{\pi_T}[f]} \geq \exp\Par{-\int_0^T \alpha_t \, \D t} \operatorname{Ent}_\pi[f].
    \end{align*}
\end{lemma}
\begin{proof}
Fix a measurable test function $f: \R^d \to \mathbb R_+$. Define for $t \geq 0$
\begin{align*}
    \nu_t(\mbf x) \propto f(\mbf x) \pi_t(\mbf x).
\end{align*}
If we consider the process $M_t \coloneqq \E_{\pi_t}[f]$, then~\eqref{eq:density_def} gives
\begin{align*}
    \D M_t = \E_{\pi_t}[f \langle\mbf x - \bs \mu_t, \D \bs W_t\rangle] = M_t \langle\mbf b(\nu_t) - \mbf b(\pi_t), \D \vW_t \rangle,
\end{align*}
and so $M_t$ is a martingale.
It\^o's lemma then tells us that
\begin{align*}
    \D (M_t \log M_t) = \frac{1}{2} M_t \norm{\mbf b(\nu_t)  - \mbf b(\pi_t)}_2^2 \, \D t + \text{martingale}.
\end{align*}
Thus, we can compute the entropy differential,
\begin{align*}
    \D \operatorname{Ent}_{\pi_t}[f] &= \D \E_{\pi_t}[f \log f] - \D (M_t \log M_t) \\
    &= -\frac{1}{2} M_t \norm{\mbf b(\nu_t)  - \mbf b(\pi_t)}_2^2 \, \D t + \text{martingale},
\end{align*}
where $\D \E_{\pi_t}[f \log f]$ is a martingale because $\pi_t$ is a martingale pointwise. Lemma~\ref{lem:entropic-stable-equiv} and \eqref{eq:ent_kl} imply
\begin{align*}
    \D \operatorname{Ent}_{\pi_t}[f] &\geq -\alpha_t M_t \KL{\nu_t}{\pi_t}\, \D t + \text{martingale} = -\alpha_t \Ent_{\pi_t}[f] + \text{martingale.}
\end{align*}
Taking expectations, and applying Gr\"onwall's inequality, we conclude that
\[
    \E\Brack{\operatorname{Ent}_{\pi_T}[f] } \geq \exp\Par{-\int_0^T \alpha_t \, \D t}\operatorname{Ent}_{\pi}[f]. \qedhere
\]
\end{proof}

\subsubsection{Log-Sobolev inequality from entropy conservation}\label{app:lsi}

We now show how to use the entropy conservation bound from Lemma~\ref{lem:conservation} to establish rapid mixing of \eqref{eq:rgd}. Recall that a transition kernel $\msf P$ with stationary measure $\pi$  satisfies a (modified) \emph{log-Sobolev inequality} (LSI) with constant $C_{\operatorname{LS}}$ if the following inequality holds:
\begin{align*}
    C_{\operatorname{LS}}(\msf P) \coloneqq 1 - \sup_{f: \Omega \to \mathbb R_+} \frac{\operatorname{Ent}_\pi[\msf P f]}{\operatorname{Ent}_\pi f}.
\end{align*}
Lower bounding the LSI constant is useful because it implies rapid mixing.
Specifically, a Markov chain with transition kernel $\sfP$, initialized at a distribution $\mu$ with $\frac \mu \pi \le \beta$ pointwise, mixes to within $\eps$ in total variation of $\pi$ in $O(\frac 1 {C_{\text{LS}}(\sfP)}\log(\frac{\log \beta}{\eps}))$ steps; see e.g., Lemma 2.4 in \cite{BlancaCPSV21}. This is because $\KL{\cdot}{\pi}$ improves by a $1 - C_{\text{LS}}(\sfP)$ factor in each step.

We next reproduce a key observation of \cite{ChenE22}: entropy conservation of a localization process, as in Lemma~\ref{lem:conservation}, lower bounds the LSI constant of the associated localization dynamics.

\begin{lemma}[{\cite[Proposition 19]{ChenE22}}]\label{lem:ent-difference}
    Assume that $\{\pi_t\}_{t \geq 0}$ is a localization process associated to $\pi$, and that $\msf P_T^\pi$ is the transition kernel \eqref{eq:kernel_loc} associated to its localization dynamics. Then,
    \begin{align*}
        C_{\operatorname{LS}}(\msf P_T^\pi) \geq \inf_{f: \Omega \to \mathbb R_+} \frac{\E[\operatorname{Ent}_{\pi_T}[f]]}{\operatorname{Ent}_\pi[f]}\,.
    \end{align*}
\end{lemma}
\begin{proof}
Throughout this proof, $\E$ with no subscript implies that the expectation is over the randomness used in defining $\pi_T$ (i.e., the localization process), and we let $\sfP \defeq \sfP_T^\pi$ for short.

    Let $\vx \in \R^d$ be arbitrary. By Jensen's inequality applied to the convex function $c \to c\log c$, with random variable $\E_{\pi_T}[f]$ and measure $\E[\frac{\pi_T(\vx)}{\pi(\vx)} \cdot]$ (as $\E[\frac{\pi_T(\vx)}{\pi(\vx)}] = 1$), we have that
    \begin{align*}
        \E\Bigl[\frac{\pi_T(\vx)}{\pi(\vx)} \E_{\pi_T}[f] \Bigr] \log \Par{\E\Bigl[\frac{\pi_T(\vx)}{\pi(\vx)} \E_{\pi_T}[f] \Bigr]} \leq \E\Bigl[\frac{\pi_T(\vx)}{\pi(\vx)} \Bigl(\E_{\pi_T}[f] \log \E_{\pi_T}[f]\Bigr)\Bigr].
    \end{align*}
    Integrating over $\vx \sim \pi$, this implies via Fubini's theorem that
    \begin{align*}
        \E_\pi[\msf P f \log \msf P f] &= \E_\pi \biggl[\E\Bigl[\frac{\pi_T}{\pi} \E_{\pi_T}[f] \Bigr]\log \E\Bigl[\frac{\pi_T}{\pi} \E_{\pi_T}[f] \Bigr]\biggr] \\
        &\leq \E_\pi\Bigl[\E\Bigl[\frac{\pi_T}{\pi} \Bigl(\E_{\pi_T}[f] \log \E_{\pi_T}[f]\Bigr)\Bigr]\Bigr] \\
        &= \E \biggl[\E_\pi\Bigl[\frac{\pi_T}{\pi} \E_{\pi_T}[f] \log \E_{\pi_T}[f]\Bigr]\biggr] =\E\Bigl[\E_{\pi_T}[f] \log \E_{\pi_T}[f]\Bigr].  
    \end{align*}
    Finally,
    \begin{align*}
        \frac{\E[\operatorname{Ent}_{\pi_T}[f]]}{\operatorname{Ent}_\pi[f]} &= \frac{\E[\E_{\pi_T}[f \log f]] - \E[\E_{\pi_T}f \log \E_{\pi_T}f]}{\operatorname{Ent}_\pi[f]} \\
        &\leq \frac{\E[\E_{\pi_T}[f \log f]] - \E_\pi[\msf P f \log \msf P f]}{\operatorname{Ent}_\pi[f]} \\
        &=\frac{\Par{\E_\pi[f \log f] - \E_\pi[f] \log \E_\pi[f]} - \Par{\E_\pi[\msf P f \log \msf P f] - \E_\pi[\sfP f] \log \E_\pi[\sfP f]}}{\operatorname{Ent}_\pi[f]} \\
        &= 1-\frac{\operatorname{Ent}_\pi[\msf Pf]}{\operatorname{Ent}_\pi[f]},
    \end{align*}
    where we used that $\pi$ is stationary for $\sfP$, as well as $\E \pi_T = \pi$, in the third line.
    This concludes the proof upon infimizing over $f$ on both sides.
\end{proof}

Our final result applies the development thus far to the localization dynamics \eqref{eq:rgd}.

\begin{theorem}[{\cite[Theorem 58]{ChenE22}}]\label{thm:ce-bound}
    If $\pi$ is $\alpha$-strongly log-concave (i.e., $-\nabla^2 \log \pi \succeq \alpha \id_d$ pointwise on $\R^d$), then the Markov chain with transition kernel $\sfP$ defined by \eqref{eq:rgd} satisfies 
    $C_{\operatorname{LS}}(\msf P) \geq \frac{\alpha}{\alpha + \eta^{-1}}$.
\end{theorem}
\begin{proof}
Throughout this proof, let $T \defeq \frac 1 \eta$.
    From the definition of the localization process \eqref{eq:tilt_def}, the localized measure $\pi_t$ is always of the form
    \begin{align*}
        \pi_t \propto \exp\Par{\inprod{\vc_t}{\cdot} - \frac t 2 \norm{\cdot}_2^2}\pi,
    \end{align*}
    and therefore $
        \nabla^2 \log \pi_t \preceq -(\alpha + t) \mbf I_d$ pointwise,
    as linear tilts do not affect the second derivative matrix.
    The famous Brascamp-Lieb inequality \cite{BrascampL76} then states that for all $\vy \in \R^d$,
    \begin{align*}
        \normop{\Cov(\mc T_{\mbf \vy} \pi_t)} \leq \frac{1}{\alpha + t}.
    \end{align*}
    Lemma~\ref{lem:tilt-stability} now implies that $\pi_t$ is $\frac{1}{\alpha + t}$-entropically stable. Then, applying Lemma~\ref{lem:conservation} gives
    \begin{align*}
        \E[\operatorname{Ent}_{\pi_T}[f]] \geq \exp\Bigl( -\int_0^T \frac{1}{\alpha + t} \, \D t\Bigr)\operatorname{Ent}_\pi[f] = \frac{\alpha}{\alpha + T} \operatorname{Ent}_\pi[f],
    \end{align*}
    for all $f: \R^d \to \R_+$. Finally, Lemma~\ref{lem:ent-difference} tells us that $C_{\operatorname{LS}}(\msf P) \geq \frac{\alpha}{\alpha + T}$,
    as desired.
\end{proof}

\subsection{The perspective from renormalization}

In this appendix, we show how the Polchinski semigroup perspective from Section~\ref{sec:renormalization} can be used to derive log-Sobolev inequalities and entropic stability estimates for the localization scheme (Definition~\ref{def:loc-proc}) induced by the stochastic localization process (Perspective~\ref{perspective:tilt}).

For consistency with Section~\ref{sec:renormalization}, in this section we let $\tau\in [0, 1]$. Recall that under the time change $t \gets \frac \tau {1 - \tau}$, Theorem~\ref{thm:tilt_renormal} shows that the renormalized measure $\pi^{\vv_\tau}_\tau$ defined in \eqref{eq:fluctuation2}, \eqref{eq:polchinski_sde} is identical in distribution to the localized measure $\pi_t$ used to define the stochastic localization scheme.

We first derive functional inequalities for the renormalized measures in Perspective~\ref{perspective:renorm}. Recall that we say a measure $\pi$ satisfies a \emph{Poincar\'e inequality} with constant $\alpha$ if for all $f \in \calF(\R^d)$,
\begin{equation}\label{eq:pi_def}\Var_\pi[f] \le \frac 1 \alpha \E_\pi\Brack{\norm{\nabla f}_2^2},\end{equation}
and similarly, $\pi$ satisfies a \emph{log-Sobolev inequality}\footnote{This is related to the modified log-Sobolev inequality in \S\ref{app:lsi}.} with constant $\alpha$ if for all $f \in \calF(\R^d)$,
\begin{equation}\label{eq:lsi_def}\Ent_\pi\Brack{f} \le \frac 2 \alpha \E_\pi\Brack{\norm{\nabla \sqrt f}_2^2}.\end{equation}
We remark that a standard linearization argument shows that a log-Sobolev inequality implies a Poincar\'e inequality with the same constant~\cite[Exercise 1.7]{Chewi25}.

\begin{proposition}\label{prop:nu-functional}
    Suppose $\pi$ satisfies a Poincar\'e (resp.\ log-Sobolev) inequality with constant $\alpha$.  Then the renormalized measure $\nu_\tau$ satisfies a Poincar\'e (resp.\ log-Sobolev) inequality with constant
    \[\frac{\alpha}{\tau(\alpha+\tau-\alpha\tau)}.\]
\end{proposition}
\begin{proof}
    From Perspective \ref{perspective:renorm}, $\nu_\tau(\vx)\propto \exp(-V_\tau(\vx)-\frac{1}{2\tau}\norm{\vx}_2^2)$, with 
    \[V_\tau(\vx)=-\log \mathbb{E}_{\vz\sim \mathcal{N}(\mathbf{0}_d, (1-\tau)\id)}[\exp(-V_1(\vx+\vz)].\]
    Then explicitly, 
    \begin{align*}
        \nu_\tau(\vx)
        &\propto \int \exp\Par{-V_1(\vx+\vz)-\frac{1}{2\tau}\norm{\vx}_2^2-\frac{1}{2(1-\tau)}\norm{\vz}_2^2}\,\dd\vz\\
        &=\int \exp\Par{\frac{1}{2}\norm{\vx+\vz}_2^2-\frac{1}{2\tau}\norm{\vx}_2^2-\frac{1}{2(1-\tau)}\norm{\vz}_2^2}\pi(\vx+\vz)\,\dd\vz\\
        &=\int \exp\Par{\frac{1}{2}\norm{\vy}_2^2-\frac{1}{2\tau}\norm{\vx}_2^2-\frac{1}{2(1-\tau)}\norm{\vx-\vy}_2^2}\pi(\vy)\,\dd\vy\\
        &=\int \exp\Par{\frac{1}{1-\tau}\inprod{\vx}{\vy}-\frac{\tau}{2(1-\tau)}\norm{\vy}_2^2-\frac{1}{2\tau(1-\tau)}\norm{\vx}_2^2}\pi(\vy)\,\dd\vy\\
        &=\int \exp\Par{-\frac{\tau}{2(1-\tau)}\norm{\frac{\vx}{\tau}-\vy}^2}\pi(\vy)\,\dd\vy.
    \end{align*}
    This is almost of the form of a Gaussian convolution.  We first define
    \[\nu_\tau'(\vx)\propto \int \exp\Par{-\frac{\tau}{2(1-\tau)}\norm{\vx-\vy}^2}\pi(\vy)\,\dd\vy.\]
    Then this is a convolution of two densities that satisfy Poincar\'e (log-Sobolev) inequalities with constants $\alpha$ and $\frac \tau {1-\tau}$, respectively.  Therefore~\cite[Proposition 2.3.8]{Chewi25}, $\nu'_\tau$ satisfies a Poincar\'e (log-Sobolev) inequality with constant
    \[\frac{1}{\frac{1}{\alpha}+\frac{1-\tau}{\tau}}=\frac{\alpha\tau}{(1-\tau)\alpha+\tau}.\]
    Now $\nu'_\tau$ and $\nu_\tau$ are equivalent up to a Lipschitz transformation of $\frac 1 \tau$; thus~\cite[Proposition 2.3.3]{Chewi25}, $\nu_\tau$ satisfies a Poincar\'e (log-Sobolev) inequality with constant 
    \[\frac{\alpha}{\tau((1-\tau)\alpha+\tau)},\]
    as desired.
\end{proof}

We now discuss what we call $\Phi$-conservation of $\pi$.  Let $\Phi: \R \to \R$ be a convex function and $f: \Omega \to \R$ an arbitrary function, and let $\Psi$ be a functional such that
\begin{equation}\label{eq:psi_def}\Psi_{\nu}[f]:=\mathbb{E}_{\nu}[\Phi(f)]-\Phi(\mathbb{E}_{\nu}[f]).\end{equation}
For example, when $\Phi(x) = x\log x$, then $\Psi = \operatorname{Ent}$, and when $\Phi(x) = x^2$, then $\Psi = \operatorname{Var}$.
From this definition and the fact that $\pi = \nu_1$ in Perspective~\ref{perspective:renorm}, we derive the decomposition
\begin{equation}\label{eq:psi-decomp}
\begin{aligned}
    \Psi_{\pi}[f] &= \E_{\nu_1}\Brack{\Phi(f)} - \Phi\Par{\E_{\nu_1}[f]}
    \\ 
    &= \E_{\nu_\tau}\Brack{\sfP_{\tau, 1} \Phi(f)} -\E_{\nu_\tau}\Brack{\Phi\Par{\sfP_{\tau, 1} f}} + \E_{\nu_\tau}\Brack{\Phi\Par{\sfP_{\tau, 1} f}} - \Phi\Par{\E_{\nu_\tau}\Brack{\sfP_{\tau, 1} f}} \\
    &= \E_{\nu_\tau}\Brack{\sfP_{\tau, 1} \Phi(f) - \Phi\Par{\sfP_{\tau, 1} f}} + \Psi_{\nu_\tau}\Brack{\sfP_{\tau, 1} f}
    \\ &=\mathbb{E}_{\nu_\tau}\left[\Psi_{\pi_\tau^{\vv}}[f]\right] + \Psi_{\nu_\tau}[\sfP_{\tau, 1}f] 
    \end{aligned}
\end{equation}
for any $\tau \in [0, 1]$, where we used Lemma~\ref{lem:semigroup} in the second line, and \eqref{eq:renorm_expect} twice in the last.
\begin{proposition}\label{prop:func_ineq_renorm}
Let $\Phi$ be a convex function, and define $\Psi$ as in \eqref{eq:psi_def}.
Following notation in Perspective~\ref{perspective:renorm}, 
assume that for all $\tau \in [0, 1]$, $\nu_\tau$ satisfies the following functional $\Psi$-inequality:
    \begin{equation}\label{eq:psi_ineq}\Psi_{\nu_\tau}[f]\leq \frac{1}{2\gamma_{\tau}}\mathbb{E}_{\nu_\tau}\left[ \Phi''(f)\norm{\nabla f}_2^2\right],\end{equation}
    for some $\{\gamma_{\tau}\}_{\tau \in [0, 1]}$.
    Then $\pi$ satisfies $\Psi$-conservation, meaning that
    \[\Par{1-\exp\Par{-\int_\tau^1 \gamma_{ \sigma}\, \D\sigma}}\Psi_\pi[f]\leq \mathbb{E}_{\nu_\tau}\left[\Psi_{\pi_\tau^{\vv}}[f]]\right.\]
\end{proposition}
\begin{proof}
    First, note that for all $\tau \in [0, 1]$,
    \begin{align*}
        \Psi_{\nu_\tau}[\sfP_{\tau, 1}f] &=\mathbb{E}_{\nu_\tau}[\Phi(\sfP_{\tau, 1}f)]-\Phi(\mathbb{E}_{\nu_\tau}[\sfP_{\tau, 1}f])\\
        &=\mathbb{E}_{\nu_\tau}[\Phi(\sfP_{\tau, 1}f)]-\Phi(\mathbb{E}_{\pi}[f])\\
        \implies  \partial_\tau \Psi_{\nu_\tau}[\sfP_{\tau, 1}f] &=\partial_\tau \mathbb{E}_{\nu_\tau}[\Phi(\sfP_{\tau, 1}f)].
    \end{align*}
    By first applying the chain rule, and then recalling the action of $\sfL_{\sig}$ given in Lemma \ref{lem:polchinski_gen},
    \begin{align*}
        \partial_\tau\mathbb{E}_{\nu_\tau}[\Phi(\sfP_{\tau, 1}f)]
        &=\mathbb{E}_{\nu_\tau}\left[\sfL_\tau(\Phi(\sfP_{\tau, 1}f))+\Phi'(\sfP_{\tau, 1} f)\frac{\partial}{\partial \tau}\sfP_{\tau, 1}f\right]\\
        &=\mathbb{E}_{\nu_\tau}\left[ \frac{1}{2}\Phi''(\sfP_{\tau, 1}f)\norm{\nabla \sfP_{\tau, 1}f}_2^2 + \Phi'(\sfP_{\tau, 1}f)\sfL_\tau \sfP_{\tau, 1}f+\Phi'(\sfP_{\tau, 1}f)\frac{\partial}{\partial \tau}\sfP_{\tau, 1}f \right]\\
        &=\frac{1}{2}\mathbb{E}_{\nu_\tau}\left[\Phi''(\sfP_{\tau, 1}f) \norm{\nabla \sfP_{\tau, 1}f}_2^2\right].
    \end{align*}
    This is also a consequence of the \textit{diffusion chain rule}~\cite[Definition 2.2.13]{Chewi25}, as the carr\'e du champ of $\sfP_{\sigma, \tau}$ is $\frac{1}{2}\inprod{\nabla}{\nabla }$.  Then, the $\Psi$-inequality \eqref{eq:psi_ineq} implies 
    \[\partial_\tau \Psi_{\nu_\tau}[\sfP_{\tau, 1}f] =\frac{1}{2}\mathbb{E}_{\nu_\tau}\left[\Phi''(\sfP_{\tau, 1}f)\norm{\nabla \sfP_{\tau, 1}f}_2^2 \right]\geq \gamma_{ \tau}\Psi_{\nu_\tau}[\sfP_{\tau, 1}f] .\]
    Defining $E(\tau):=\Psi_{\nu_\tau}[\sfP_{\tau, 1}f]$, we have
    \[\log E(1)-\log E(\tau)=\int_{\tau}^1 \frac{E'(\sigma)}{E(\sigma)}\D\sigma\geq \int_\tau^1 \gamma_\sigma \D \sigma.\]
    Rearranging, 
    \[\Psi_{\nu_\tau}[\sfP_{\tau, 1}f]\leq \exp\Par{\int_1^\tau \gamma_{ \sigma}\D\sigma}\Psi_{\nu_1}[f].\]
    Now using $\nu_1=\pi$ and \eqref{eq:psi-decomp} we obtain the conclusion.
\end{proof}

We now derive the consequences of Propositions~\ref{prop:nu-functional} and~\ref{prop:func_ineq_renorm} for $\pi$ satisfying a suitable functional inequality.

\begin{corollary}\label{cor:renorm-constants}
    Suppose $\pi$ satisfies a Poincar\'e inequality with constant $\alpha$.  Then for all $\tau \in [0, 1]$, we have the variance conservation bound
    \[
        \frac{\alpha}{\alpha + \frac \tau {1 - \tau}}\Var_\pi[f]\leq \mathbb{E}_{\nu_\tau}\left[\Var_{\pi_\tau^{\vv}}[f]\right].
    \]
    Similarly, suppose $\pi$ satisfies a log-Sobolev inequality with constant $\alpha$.  Then for all $\tau \in [0, 1]$, we have the entropy conservation bound
\[\frac{\alpha}{\alpha + \frac \tau {1 - \tau}}\Ent_\pi[f]\leq \mathbb{E}_{\nu_\tau}\left[\Ent_{\pi_\tau^{\vv}}[f]\right].\]
\end{corollary}
\begin{proof}
We prove the statement for variance conservation; the proof for entropy conservation follows analogously.  If $\pi$ satisfies a Poincar\'e inequality with constant $\alpha$, then Proposition \ref{prop:nu-functional} establishes that $\nu_\tau$ satisfies a Poincar\'e inequality with constant 
\[\gamma_\tau:=\frac{\alpha}{\tau(\alpha+\tau-\alpha\tau)}.\]
By \eqref{eq:pi_def} $\nu_\tau$ satisfies \eqref{eq:psi_ineq} with this $\gamma_\tau$ and $\Psi=\operatorname{Var}$.  Applying Proposition \ref{prop:func_ineq_renorm} the statement is proven.
\end{proof}

Notably, under the reparameterization in Theorem~\ref{thm:tilt_renormal}, i.e., the time-change $t \gets \frac \tau {1 - \tau}$ and equivalence of localized measures, Theorem~\ref{thm:ce-bound} and Corollary~\ref{cor:renorm-constants} yield the same bound. The generality of Proposition~\ref{prop:func_ineq_renorm} further  yields a contraction $\chi^2(\cdot \| \pi)$ along the dynamics \eqref{eq:rgd}, similar to our derivation in Appendix~\ref{app:lsi} for $\KL{\cdot}{\pi}$. We defer a formal proof that variance conservation implies $\chi^2$ mixing to Proposition 19, \cite{ChenE22}.

\subsection{The persective from time reversal}\label{sec:perspective-adjointheat}

In this appendix, we show how our perspective in Section~\ref{sec:diffusion} of the stochastic localization process as a backwards heat flow also yields quantitative estimates on how quickly the dynamics \eqref{eq:rgd} converge. The computations here are replicated from~\cite{Klartag2021Spectral, ChenCSW22}. 

Formally, Section~\ref{sec:diffusion} considers a time reversal of the OU process, $\dd \vx_t = -\vx_t \dd t + \sqrt 2 \, \dd \vB_t$. Here, we instead reverse the \emph{heat flow}, which advances $\dd \vx_t = \dd \vB_t$, which behaves the same under reparameterization. If $\vx_0 \sim \pi \propto \exp(-V)$, we then have that $\vx_t = \pi \ast \gamma_t$, where $\gamma_t$ is the density of the Gaussian $\Nor(\0_d, t \id_d)$. We write this convolved measure as $\pi \msf Q_t$. 

The forwards and backwards processes we consider in this appendix are as follows, defined in a time range $t \in [0, \eta]$ for some $\eta \in (0, \infty)$:
\begin{align}
    \D \mathbf x_t &=  \D \vB_t \label{eq:forward-heat}\\
    \D \backvx_t &= \nabla \log \pi \msf Q_{\eta-t}(\backvx_t) \, \D t + \, \D \vW_t. \label{eq:backward-heat}
\end{align}
Here, the form of \eqref{eq:backward-heat} follows directly from an analogous calculation as used in Lemma~\ref{lem:timechange}.

The reason we are interested in the dynamics \eqref{eq:forward-heat}, \eqref{eq:backward-heat} is because they exactly capture the restricted Gaussian dynamics \eqref{eq:rgd}. Indeed, note that advancing the forward process \eqref{eq:forward-heat} for time $\eta$ from $\vx_0 \gets \vx$ implements the first step of \eqref{eq:rgd}. Similarly, advancing the backwards process \eqref{eq:backward-heat} for time $\eta$ from $\backvx_0 \gets \vx_\eta = \vy$ implements posterior sampling, which gives the second step of \eqref{eq:rgd} by using a similar calculation as in Theorem~\ref{thm:tilt_posterior}.

We summarize some key facts about \eqref{eq:forward-heat}, \eqref{eq:backward-heat}: the proofs are routine applications of Lemma~\ref{lem:fokker_planck} and integration by parts to compute generators as adjoints of time evolutions.

\begin{lemma}\label{lem:fokker-planck-rgo}
The generators and carr\'e du champs of \eqref{eq:forward-heat} and \eqref{eq:backward-heat} are as follows.
    \begin{enumerate}[label=(\roman*)]
        \item If $\vx_0 \sim \mu_0$ and $\mu_t \defeq \textup{Law}(\vx_t)$ where $\vx_t$ follows \eqref{eq:forward-heat}, then 
        \[\partial_t \mu_t = \half \Delta \mu_t.\]
        The generator $\msf L f = \frac{1}{2} \Delta f$ is time-invariant, and the carr\'e du champ is $\msf \Gamma(f,g) = \frac{1}{2}\langle \nabla f, \nabla g\rangle$.

        \item If $\backvx_0 \sim \backmu_0$ and $\backmu_t \defeq \textup{Law}(\backvx_t)$ where $\backvx_t$ follows \eqref{eq:backward-heat}, then 
        \[\partial_t \backmu_t = \half \Delta \backmu_t - \nabla \cdot(\backmu_t \nabla \log \pi \sfQ_{\eta - t}).\] The generator $\msf L_t^\leftarrow f = \frac{1}{2} \Delta f + \langle \nabla \log \pi \msf Q_{\eta-t}, \nabla f\rangle$ is time-dependent, and the carr\'e du champ $\msf \Gamma_t(f,g) \equiv \msf \Gamma(f,g) = \frac{1}{2}\langle \nabla f, \nabla g\rangle$ is time-invariant.    \end{enumerate}
\end{lemma}

We also require a technical lemma about the time change in an $f$-divergence between two densities going simultaneous evolution. We defer a proof to the excellent reference \cite{Chewi25}.

\begin{lemma}[{Simultaneous heat flow; adapted from~\cite[Theorem 8.3.1]{Chewi25}}]\label{lem:simultaneous-heat}
    Consider two measures co-evolving according the same dynamics,
    \begin{align*}
        \partial_t \mu_t = (\msf L_t)^* \mu_t\,, \qquad \partial_t \nu_t = (\msf L_t)^* \nu_t.
    \end{align*}
    Then, assuming $(\msf L_t)^*$ satisfies a technical condition\footnote{This is again the diffusion chain rule~\cite[Definition 2.2.13]{Chewi25} discussed in the proof of Proposition~\ref{prop:func_ineq_renorm}, which is true for our applications of interest (and carr\'e du champs of the form $\msf \Gamma(f,f) = c \norm{\nabla f}_2^2$).} and $\msf \Gamma_t$ is its associated carr\'e du champ operator, we have for an $f$-divergence $\mc D_f(\mu \mmid \nu) \coloneqq \int f(\frac{\mu}{\nu}) \, \D \nu$,
    \begin{align*}
        \partial_t \mc D_f(\mu_t \mmid \nu_t) = - \int f''\Bigl(\frac{\mu_t}{\nu_t} \Bigr) \msf \Gamma_t\Bigl( \frac{\mu_t}{\nu_t}, \frac{\mu_t}{\nu_t}\Bigr)\, \D \nu_t\,.
    \end{align*}
\end{lemma}

Finally, we require an ancilliary lemma on convolving strongly log-concave functions.

\begin{lemma}\label{lem:slc_convolve}
Let $\mu: \R^d \to \R_{\ge 0}$ be $\alpha$-strongly log-concave and let $\nu: \R^d \to \R_{\ge 0}$ be $\beta$-strongly log-concave, for some $\alpha, \beta > 0$. Then their convolution $\mu \ast \nu$ is $\frac {\alpha\beta}{\alpha+\beta}$-strongly log-concave.
\end{lemma}
\begin{proof}
By definition, $\gamma$-strong log-concavity of a function $f: \R^d \to \R_{\ge 0}$ is equivalent to $f \exp(\frac \gamma 2 \norm{\cdot}_2^2)$ being log-concave. Now, define the function
\[f(\vx, \vy) \defeq \mu(\vy) \nu(\vx - \vy) \exp\Par{\frac \gamma 2 \norm{\vx}_2^2}.\]
We claim that $f$ is log-concave if $\gamma \le \frac{\alpha\beta}{\alpha + \beta}$. This holds because
\begin{gather*}-\nabla^2 \log \mu(\vy) - \nabla^2 \log \nu(\vx - \vy) - \nabla^2\Par{\frac{\gamma}{2}\norm{\vx}_2^2} \\
\succeq \begin{pmatrix} \0_{d \times d} & \0_{d \times d} \\ \0_{d \times d} & \alpha \id_d \end{pmatrix} + \begin{pmatrix} \beta \id_d & -\beta \id_d \\ -\beta \id_d & \beta \id_d \end{pmatrix} - \gamma \begin{pmatrix} \id_d & \0_{d \times d} \\ \0_{d \times d} & \0_{d \times d}\end{pmatrix}
= \begin{pmatrix} (\beta - \gamma) \id_d &-\beta \id_d\\ -\beta \id_d & (\alpha + \beta) \id_d \end{pmatrix},
\end{gather*}
and we can verify that the last matrix is positive semidefinite by checking its determinant. Finally, the conclusion holds because the $\vx$-marginal of $f$ is log-concave by the Pr\'ekopa-Leindler inequality, which simplifies to the convolution of $\mu$ and $\nu$ being $\gamma$-strongly log-concave.
\end{proof}

To interpret Lemma~\ref{lem:slc_convolve}, recall that classical Bakry-\'Emery theory shows that strong log-concavity implies a log-Sobolev inequality \eqref{eq:lsi_def}. Combining this fact with Lemma~\ref{lem:slc_convolve} thus gives a log-Sobolev inequality for densities which result from applying a heat kernel to a strongly log-concave measure. 

Finally, we are ready to state our main result, a direct bound on the contraction of \eqref{eq:rgd} to its stationary distribution in the KL divergence. Note that the estimate established here is in fact slightly stronger than that concluded by Theorem~\ref{thm:ce-bound}. 

\begin{theorem}\label{thm:heat-flow-cvg}
    Let $\pi$ be $\alpha$-strongly log-concave, and consider following the dynamics \eqref{eq:rgd} from $\vx \sim \mu$, where we define $\mu' \defeq \textup{Law}(\vx')$. 
    Then,
    \begin{align*}
        \mathrm{KL}(\mu' \mmid \pi) \leq \frac{\mathrm{KL}(\mu \mmid \pi)}{(1+\alpha \eta)^2}.
    \end{align*}
\end{theorem}
\begin{proof}
As remarked in our discussion following \eqref{eq:forward-heat}, \eqref{eq:backward-heat}, an equivalent way to simulate \eqref{eq:rgd} is by running \eqref{eq:forward-heat} from $\vx_0 \gets \vx$, and then running \eqref{eq:backward-heat} from $\backvx_0 \gets \vx_\eta$, where we output $\vx' \gets \backvx_0$.

    In the sequel, let $\pi \msf Q_t$ denote the law of $\vx_t$ if $\vx_0 \sim \pi$; likewise, it is the law of $\mathbf x_{\eta-t}^\leftarrow$ if $\mbf x_0^\leftarrow \sim \pi \msf Q_\eta$, which follows from the time reversal property. We also define $\{\mu_t\}_{t \in [0, \eta]}$, $\{\backmu_t\}_{t \in [0, \eta]}$ as in Lemma~\ref{lem:fokker-planck-rgo}, where $\mu_0 \defeq \mu$. We consider the forward and backward heat flows separately.
    \paragraph{Forward heat:}
    Applying Lemma~\ref{lem:simultaneous-heat} to the $\mathrm{KL}$ divergence, which is an $f$-divergence with $f(x) = x \log x$, and with the carr\'e du champ $\msf \Gamma_t(f, f) = \half \norm{\nabla f}_2^2$ from Lemma~\ref{lem:fokker-planck-rgo}(i), we have
    \begin{align*}
        \partial_t \mathrm{KL}(\mu_t \mmid \pi \msf Q_t) = -\frac{1}{2} \int \frac{\pi \msf Q_t}{\mu_t} \norm{\nabla \frac{\mu_t}{\pi \msf Q_t}}_2^2 \, \D \pi \msf Q_t \geq - \frac{\alpha}{1+\alpha t} \mathrm{KL}(\mu_t \mmid \pi \msf Q_t),   
    \end{align*}
    where the inequality follows from Lemma~\ref{lem:slc_convolve} and the fact that a $\gamma$-strongly log-concave measure satisfies a log-Sobolev inequality \eqref{eq:lsi_def} with constant $\gamma$. This implies via Gronwall's inequality that
    \begin{align}\label{eq:forward-decay}
        \mathrm{KL}(\mu_\eta \mmid \pi \msf Q_\eta) \leq \exp\Bigl(-\int_0^\eta \frac{\alpha}{1+\alpha t} \, \D t \Bigr) \mathrm{KL}(\mu_0 \mmid \pi) = \frac{\mathrm{KL}(\mu_0 \mmid \pi)}{1+\alpha \eta} .
    \end{align}

    \paragraph{Backward heat:}
    Applying Lemma~\ref{lem:simultaneous-heat} and Lemma~\ref{lem:fokker-planck-rgo}(ii), we have completely analogously that
    \begin{align*}
        \partial_t \mathrm{KL}(\mu_t^\leftarrow \mmid \pi \msf Q_{\eta-t}) = \frac{1}{2} \int \frac{\pi \msf Q_{\eta - t}}{\backmu_t} \norm{\nabla \frac{\backmu_t}{\pi \msf Q_{\eta - t}}}_2^2 \, \D \pi \msf Q_{\eta - t}\geq -\frac{\alpha}{1+\alpha (\eta-t)} \mathrm{KL}(\mu_0^\leftarrow \mmid \pi \msf Q_{\eta-t}),
    \end{align*}
    and as a result, Gronwall's inequality yields
    \begin{align}\label{eq:backward-decay}
        \partial_t \mathrm{KL}(\mu_\eta^\leftarrow \mmid \pi) \leq \exp\Bigl(-\int_0^\eta \frac{\alpha}{1+\alpha(\eta-t)} \, \D t \Bigr) \mathrm{KL}(\mu_0^\leftarrow \mmid \pi \msf Q_\eta) = \frac{\mathrm{KL}(\mu_0^\leftarrow \mmid \pi \msf Q_\eta)}{1+\alpha \eta}.
    \end{align}

    Finally, combining~\eqref{eq:forward-decay} and~\eqref{eq:backward-decay} completes the proof.
\end{proof}

Lastly, we note that Theorem~\ref{thm:heat-flow-cvg} can be strengthened in a few ways. A fortiori, Lemma~\ref{lem:simultaneous-heat} allows us to write a similar result for R\'enyi divergences, and direct use of functional inequalities allows extensions beyond the strong log-concavity. We defer more discussion to~\cite[Chapter 8]{Chewi25}.